\documentclass[11pt]{amsart}
\usepackage{a4wide}
\usepackage{dsfont} 
\usepackage{amssymb}
\usepackage{pstricks}
\usepackage{graphicx}
\usepackage{MnSymbol,wasysym}

\usepackage{amsmath}
\usepackage{amstext}
\usepackage{amsfonts}
\usepackage{amssymb}
\usepackage{amsthm}
\usepackage{amsrefs}

\newtheorem{proposition}{Proposition}[section]
\newtheorem{theorem}[proposition]{Theorem}
\newtheorem{lemma}[proposition]{Lemma}
\newtheorem{corollary}[proposition]{Corollary}

\theoremstyle{remark}
\newtheorem{definition}[proposition]{Definition}
\newtheorem{remark}[proposition]{Remark}
\newtheorem{example}[proposition]{Example}

\newcommand{\cst}{\ifmmode\mathrm{C}^*\else{$\mathrm{C}^*$}\fi}
\newcommand{\wst}{\ifmmode\mathrm{W}^*\else{$\mathrm{W}^*$}\fi}

\newcommand{\CC}{\mathbb{C}}
\newcommand{\tens}{\otimes}
\newcommand{\ot}{\otimes}

\newcommand{\wot}{\overline{\otimes}}
\newcommand{\id}{\textup{id}}
\newcommand{\Dom}{\textup{Dom}}
\newcommand{\Lin}{\textup{Lin}\,}
\newcommand{\comp}{\!\circ\!}
\newcommand{\I}{\mathds{1}}
\newcommand{\hh}[1]{\widehat{#1}}
\newcommand{\GG}{\mathbb{G}}
\newcommand{\RR}{\mathbb{R}}
\newcommand{\KK}{\mathbb{K}}

\newcommand{\HH}{\mathbb{H}}

\newcommand{\Jnd}{\mathcal{I}}
\newcommand{\vtens}{\,\bar{\tens}\,}
\newcommand{\uu}{{\scriptscriptstyle\mathrm{u}}}

\newcommand{\Hil}{\sH}
\newcommand{\sA}{\mathsf{A}}
\newcommand{\sB}{\mathsf{B}}
\newcommand{\sC}{\mathsf{C}}

\newcommand{\sM}{\mathsf{M}}
\newcommand{\sN}{\mathsf{N}}
\newcommand{\sH}{\mathsf{H}}
\newcommand{\sK}{\mathsf{K}}
\newcommand{\sL}{\mathsf{L}}
\newcommand{\sZ}{\mathsf{Z}}
\newcommand{\sV}{\mathsf{V}}

\newcommand{\ww}{\mathrm{W}}
\newcommand{\WW}{{\mathds{V}\!\!\text{\reflectbox{$\mathds{V}$}}}}
\newcommand{\Ww}{\mathds{W}}
\newcommand{\wW}{\text{\reflectbox{$\Ww$}}\:\!}

\newcommand{\hGG}{\hh{\GG}}
\newcommand{\hHH}{\hh{\HH}}
\newcommand{\hKK}{\hh{\KK}}

\DeclareMathOperator{\C}{C}
\DeclareMathOperator{\D}{D}

\DeclareMathOperator{\M}{M}
\DeclareMathOperator{\Mor}{Mor}
\DeclareMathOperator{\B}{B}
\DeclareMathOperator{\K}{\mathcal{K}}
\DeclareMathOperator{\cK}{\mathcal{K}}
\DeclareMathOperator{\Linf}{L^\infty\!\!\;}
\DeclareMathOperator{\Lone}{L^1\!\!\;}
\DeclareMathOperator{\linf}{\ell^\infty\!\!\;}
\DeclareMathOperator{\Ltwo}{L^2\!\!\;}

\numberwithin{equation}{section}

\newenvironment{rlist}
{

\begin{enumerate}}
{\end{enumerate}}

\begin{document}

\author{Mehrdad Kalantar}
\address{Department of Mathematics, University of Houston, Houston, TX 77204, USA}
\email{kalantar@math.uh.edu}

\author{Pawe{\l} Kasprzak}
\address{Department of Mathematical Methods in Physics, Faculty of Physics, University of Warsaw, Poland}
\email{pawel.kasprzak@fuw.edu.pl}

\author{Adam Skalski}
\address{Institute of Mathematics of the Polish Academy of Sciences, ul. Sniadeckich 8, 00-656 Warszawa, Poland}
\email{a.skalski@impan.pl}

\begin{abstract}
The notion of an open quantum subgroup of a locally compact quantum group is introduced and given several equivalent characterizations in terms of group-like projections, inclusions of quantum group $\C^*$-algebras and properties of respective quantum homogenous spaces. Open quantum subgroups are shown to be closed in the sense of Vaes and normal open quantum subgroups are proved to be in 1-1 correspondence with normal compact quantum subgroups of the dual quantum group.
\end{abstract}

\title{Open quantum subgroups of locally compact quantum groups}



\maketitle

The theory of locally compact quantum groups, formulated in the language of operator algebras, is a rapidly developing field closely related to abstract harmonic analysis, and with various connections to noncommutative geometry, quantum probability and other areas of `noncommutative' mathematics. A \emph{locally compact quantum group} $\GG$ is a virtual object, studied via its `algebras of functions': a $\C^*$-algebra $\C_0(\GG)$, playing the role of the algebra of continuous functions on $\GG$ vanishing at infinity, and a von Neumann algebra $L^{\infty}(\GG)$, viewed as the algebra of essentially bounded measurable functions on $\GG$; both of these are equipped with coproducts, operations encoding the `multiplication operation' of $\GG$.
It is often essential to study both of the avatars of $\GG$ mentioned above (as well as the universal counterpart of $\C_0(\GG)$, the $\C^*$-algebra $\C_0^u(\GG)$) at the same time (see for example \cite{Kasprzakhomogen}, \cite{Vaes-induction}). This will also be the case in this article.

In recent years we have seen an increased interest in the notion of \emph{morphisms between locally compact quantum groups} $\HH$ and $\GG$. These also have various incarnations
(see \cite{SLW12}), one of which is given by $\C^*$-algebra morphisms $\pi \in \Mor (\C_0^u(\GG), \C_0^u(\HH))$ intertwining the respective coproducts.
It is then natural to consider what it means that a given morphism between $\HH$ and $\GG$ has closed image
and moreover is a homeomorphism onto this image -- in other words, identifies $\HH$ with a closed (quantum) subgroup of $\GG$.
This problem was studied in depth in the article \cite{DKSS}, where two alternative definitions, due respectively to Vaes and Woronowicz,
 formulated respectively in the  von Neumann algebraic and $\C^*$-algebraic language,
were compared and interpreted in several special cases.
Among the main results of \cite{DKSS} was the proof that
a closed quantum subgroup in the sense of Vaes is always a closed quantum subgroup in the sense of Woronowicz,
and in fact in many cases (classical, dual to classical, compact, discrete) the two definitions coincide.
Understanding the concept of a closed quantum subgroup forms a very natural step in the development of the theory
and was key in building and applying the induction theory for representations of quantum groups (as formulated in \cite{KustermansInduced}, \cite{Vaes-induction}).

In the current work, which can be naturally viewed as a continuation of \cite{DKSS}, we address the problem of identifying the notion of an \emph{open quantum subgroup}. Classically, an open subgroup $H$ of a locally compact group $G$ is automatically closed, a fact that leads to simplification of many questions concerning relations between representation theoretical and harmonic analytical properties of $H$ and $G$.
Perhaps, the main reason for the latter is that openness of $H$ in $G$ implies (and in fact is equivalent to) not only the compatibility of the respective topologies, but also the measure structures, given by the appropriate Haar measures.
This is in fact the starting point of our considerations, when we declare that $\HH$ is (isomorphic to) an open quantum subgroup of $\GG$ if there exists a surjective von Neumann algebra morphism $\pi: L^{\infty}(\GG) \to L^{\infty}(\HH)$ intertwining the respective coproducts. We then show that open quantum subgroups are automatically closed in the sense of Vaes, that the related topological quantum homogeneous space $\GG/ \HH$ has a simple description and that open quantum subgroups of $\GG$ can be equivalently characterised by group-like projections in $L^{\infty}(\GG)$ (the last result was noted for classical locally compact groups by Landstad and Van Daele in \cite{LandstadVD} -- see also \cite{LandstadVD2} for some related algebraic quantum group facts). Establishing these results, almost trivial in the classical setting, in the quantum world requires significant effort. We also extend the classical (and this time already there non-obvious) theorem of Bekka, Kaniuth, Lau and Schlichting \cite{BKLS} saying that a closed subgroup $H$ of $G$ is open in $G$ if and only if the natural $C^*$-morphism in  $ \Mor (\C^*(H), \C^*(G))$ maps injectively $\C^*(H)$ into $\C^*(G)$ (as opposed to its multiplier algebra). 
Finally we show that for \emph{normal} open quantum subgroups the situation is particularly satisfactory -- they lie in 1-1 correspondence with normal compact quantum subgroups of the dual quantum group.

Finally we note that one of our initial motivations for developing the concept of open quantum subgroups was to study the relations between the theory of induced representations of locally compact quantum groups as developed by Kustermans and Vaes respectively in \cite{KustermansInduced} and \cite{Vaes-induction} and the abstract \cst-theory of induced representations developed by Rieffel in \cite{rief1}. We will address this in the forthcoming work \cite{KKSinduced}.

The detailed plan of the article is as follows: in the first, preliminary section, we introduce the notations and terminology related to locally compact quantum groups and prove a few technical lemmas. Section 2 defines open quantum subgroups and shows that if $\HH$ is an open quantum subgroup of $\GG$ then we have a natural realisation of $\Linf(\HH)$ inside $\Linf(\GG)$. In Section 3 we prove that open quantum subgroups are closed in the sense of Vaes and provide various descriptions of the quotient algebra $\Linf (\GG/\HH)$. Section 4 establishes an equivalence between open quantum subgroups and group-like projections in $\Linf(\GG)$. In Section 5 open quantum subgroups are characterised via the behaviour of representations, understood as an inclusion of the corresponding group \cst-algebras. In Section 6 we return to study of the quantum homogeneous space $\GG/\HH$, showing that openness of $\HH$ is under some technical assumptions equivalent to discreteness of $\GG/\HH$ and providing an explicit description of the algebra $\C_0(\GG/\HH)$. Finally in Section 7 we prove that normal open quantum subgroups of $\GG$ are in a 1-1 correspondence with normal compact quantum subgroups of $\hGG$.

\vspace*{0.5cm}

\noindent
{\bf Acknowledgement.}\ We thank Piotr So{\l}tan for many helpful discussions and the anonymous referee for a careful reading of our manuscript.
AS  was partially supported by the NCN (National Centre of Science) grant
2014/14/E/ST1/00525. PK was partially supported by the NCN (National Centre of Science) grant
 2015/17/B/ST1/00085.

\section{Notation and preliminaries}

We will follow closely the notations of \cite{DKSS} (see also \cite{KaspSol}). All scalar products  will be linear on the right. The symbol $\ot$ will denote the tensor product of maps and minimal spatial tensor product of $\cst$-algebras, $\wot$ will denote the ultraweak tensor product of von Neumann algebras. Given two \cst-algebras $\sA$ and $\sB$, a \emph{morphism} from $\sA$ to $\sB$ is a $*$-homomorphism $\Phi$ from $\sA$ into the \emph{multiplier algebra} $\M(\sB)$ of $\sB$, which is \emph{non-degenerate}, i.e.~the set $\Phi(\sA)\sB$ of linear combinations of products of the form $\Phi(a)b$ ($a\in\sA$, $b\in\sB$) is dense in $\sB$. The set of all morphisms from $\sA$ to $\sB$ will be denoted by $\Mor(\sA,\sB)$. The non-degeneracy of morphisms ensures that each $\Phi\in\Mor(\sA,\sB)$ extends uniquely to a unital $*$-homomorphism $\M(\sA)\to\M(\sB)$ which we will usually denote by the same symbol and use  implicitly when composing the morphisms. On the multiplier \cst-algebras we will occasionally use apart from the norm topology also the \emph{strict topology}.
For a Hilbert space $\sH$ the \cst-algebra of compact operators on $\sH$ will be denoted by $\K(\sH)$. We say that a \cst-subalgebra $\sB$ of a \cst-algebra $\sA$ is \emph{non-degenerate} if the inclusion map is a non-degenerate morphism; we also assume that all representations of \cst-algebras on Hilbert spaces are non-degenerate. Further if $\sA$ and $\sC$ are \cst-algebras we say that $\sA$ is \emph{generated by $T\in\M(\sC\tens\sA)$} if for any Hilbert space $\sH$, any representation $\rho$ of $\sA$ on $\sH$ and any \cst-algebra $\sB$ represented on $\sH$ the condition that $(\id\tens\rho)(T)\in\M(\sC\tens\sB)$ implies that $\rho\in\Mor(\sA,\sB)$.

For operators acting on tensor products of spaces we will use the familiar leg notation: so for example if $V$ is a vector space and $T\in L( V^{\ot 2})$ then, depending on which legs of the triple tensor product we want to act with $T$, we have the natural operators $T_{12}, T_{13}, T_{23} \in L( V^{\ot 3})$ (the notation will be also applied in a formally more complicated case of completed tensor products). Tensor flip between algebras will be denoted by $\sigma$, and between Hilbert spaces by $\Sigma$. If $X$ is a subset of a Banach space $\sV$, by $\overline{\Lin}X$ we mean the closed linear span of $X$. If $\xi, \eta$ are vectors in a Hilbert space $\sH$, the symbol $\omega_{\xi, \eta}$ will denote the functional $T \mapsto \langle \xi, T \eta \rangle$ on $B(\sH)$, with $\omega_{\xi}:=\omega_{\xi, \xi}$.

For a one-parameter family of automorphisms $(\gamma_t)_{t \in \RR}$ of a von Neumann algebra $\sM$ (which will always be assumed to be pointwise weak$^*$-continuous) we will use the standard notations for densely defined operators $\gamma_z, z \in \CC$ (see \cite[Subsetion 4.3]{KusNotes}). For a (normal, semifinite, faithful) weight $\phi$ on a von Neumann algebra $\sM$ we denote the left ideal of `square-integrable' elements by $\mathcal{N}_{\phi} := \{x\in\sM:\phi(x^*x)<\infty\}$ (we will also use the same notation for weights on a \cst-algebra). We will also use at a certain point \emph{slice maps for weights}, as discussed for example in Section 1.5 of \cite{KV}.

\subsection{Locally compact quantum groups -- basic facts}
Throughout the paper symbols $\GG$, $\HH$ will denote \emph{locally compact quantum groups} in the sense of Kustermans and Vaes (\cite{KV}) -- we refer the reader to the latter paper, as well as \cite{KusNotes} and \cite{DKSS} for detailed definitions of the objects to be introduced below (note however that we stick to the conventions of the last of these three sources). A locally compact quantum group (often simply called quantum group in what follows) $\GG$ is defined in terms of a von Neumann algebra $L^{\infty}(\GG)$ equipped with a unital, normal coassociative *-homomorphism $\Delta:L^{\infty}(\GG)\to L^{\infty}(\GG) \wot L^{\infty}(\GG)$, called the \emph{coproduct} or \emph{comultiplication}. The symbols $\varphi$ and $\psi$ will denote respectively \emph{left} and \emph{right invariant Haar weights} of $\GG$, which are unique up to a positive scalar multiple, and $L^2(\GG)$ will denote the GNS Hilbert space of the right Haar weight $\psi$ (on which $L^{\infty}(\GG)$ acts). We will also frequently use
the corresponding \cst-algebra of ``continuous functions on $\GG$ vanishing at infinity'', $\C_0(\GG)\subset L^{\infty}(\GG)$. The comultiplication $\Delta$ restricts to a (still coassociative) morphism $\Delta\in\Mor\bigl(\C_0(\GG),\C_0(\GG)\tens\C_0(\GG)\bigr)$. Finally we have the universal object related to $\GG$, i.e.\ a \cst-algebra which we will denote by $\C_0^\uu(\GG)$, endowed with a comultiplication $\Delta^\uu\in\Mor\bigl(\C_0^\uu(\GG),\C_0^\uu(\GG)\tens\C_0^\uu(\GG)\bigr)$,  introduced and studied in \cite{Johanuniv}. We have a canonical surjective \emph{reducing morphism} $\Lambda \in \Mor (\C_0^u(\GG), \C_0(\GG))$, intertwining the respective coproducts. If $\Lambda$ is injective, we say that $\GG$ is \emph{coamenable}.

A fundamental object in the study of $\GG$ turns out to be the \emph{right multiplicative unitary} $W\in B(L^2(\GG) \ot \Ltwo(\GG))$, which satisfies the pentagonal equation
$W_{12} W_{13} W_{23} = W_{23} W_{12}$. In fact $W$ determines $\GG$ completely, as we have on one hand the equality: $L^{\infty}(\GG)=\{(\omega \ot \id)(W): \omega \in B(L^2(\GG))_*\}''$, and on the other $W$ implements the coproduct:
\[\Delta(x) =W(x \ot \I) W^*, \;\;\; x \in L^{\infty}(\GG).\]
We also have the equality $\C_0(\GG)=\overline{\Lin}\{(\omega \ot \id)(W): \omega \in B(L^2(\GG))_*\}$.
The symbols $S$ and $R$ will respectively denote the \emph{antipode} (a densely defined operator on $L^{\infty}(\GG)$) and the \emph{unitary antipode} (a bounded anti-isomorphism of $L^{\infty}(\GG)$). Both of these have natural `universal' versions, acting on (subsets) of $\C_0^u (\GG)$. The relation between $R$ and $S$ is best described via the \emph{scaling automorphism group} $(\tau_t)_{t \in \RR}$, which is a particular uniquely determined one-parameter group of automorphisms of $\Linf(\GG)$:
$S=R \circ \tau_{-\frac{i}{2}}$.
For an n.s.f.\ weight $\phi$ on a \cst-algebra or a von Neumann algebra
we denote by $(\sigma_t^\phi)_{t \in \RR}$ and $J^\phi$, the corresponding modular automorphism group and modular conjugation, respectively.

We have the intertwining relation
\[\label{delta(sigma)}
\Delta\circ\sigma_t^{\varphi} = (\tau_t\otimes\sigma^{\varphi}_t)\circ\Delta, \;\;\; t \in \RR.\]
The antipode is determined uniquely by its \emph{strong left invariance}, which reads
\[S((\id\otimes\varphi)(\Delta(a^*)(\I\otimes b) )) = (\id\otimes\varphi)((\I\otimes a^*)\Delta^\GG(b)),\;\;\; a,b\in\mathcal{N}_{\varphi}.\]
If $S=R$ we say that  the quantum group $\GG$ is of \emph{Kac type}.
In general the antipode connects the left and right Haar weights, via the \emph{scaling constant} $\lambda>0$: $\varphi \circ S = \lambda^{\frac{i}{2}} \psi$.
Also, the left and right Haar weights are related thorough the \emph{modular element} $\delta$, an unbounded operator affiliated with $\C_0(\GG)$,
by means of a Radon-Nikodym theorem: $\varphi = \psi(\delta^{1/2} \cdot \delta^{1/2})$.

The predual of $\Linf(\GG)$ is denoted, by the analogy with the classical case, $\Lone(\GG)$. It is a Banach algebra with respect to the convolution product given by the pre-adjoint map of the comultiplication. In general $\Lone(\GG)$ is not a $^*$-algebra, but it contains a subalgebra $\Lone_{\#} (\GG):=\{ f \in \Lone(\GG): f \circ S: \Dom (S) \to \CC \textup{ is bounded}\}$, which is equipped with a natural involution $f^\#:= \overline{f \circ S}$, where on the right-hand side we naturally mean the bounded extension of $f\circ S$. The `density conditions' (which are key to the development of the \cst-algebraic approach to topological quantum groups) mean that the sets $\Lin \Delta(\Linf(\GG)) (\Linf (\GG) \ot \I)$ and $\Lin \Delta(\Linf(\GG)) (\I \ot \Linf (\GG))$ are weak$^*$-dense in $\Linf(\GG) \wot \Linf(\GG)$.
Further  $\GG$ is called \emph{regular} if the norm closure of the set $\{ (\omega \ot \id)(\Sigma W): \omega \in B(\Ltwo(\GG))_*\}$ is equal to $\cK(\Ltwo(\GG))$ (this is a `flipped' version of the definition given in \cite{BaajSkandalis} -- they are both easily seen to be equivalent to the statement $\C_0(\GG) \C_0(\hh{\GG}) = \K(\Ltwo(\GG))$, see the next paragraph). A \emph{(unitary) representation of $\GG$} on a Hilbert space $\sH$ is a unitary $\mathcal{V} \in \M (\C_0(\GG) \ot \cK(\sH))$ such that $(\Delta \ot \id)(\mathcal{V})= \mathcal{V}_{12} \mathcal{V}_{13}$.

The multiplicative unitary $W$ allows for a very simple description of the \emph{dual locally compact quantum group} of $\GG$, which we will denote by $\hh{\GG}$: we have $L^{\infty}(\hGG)=\{(\id \ot \omega )(W): \omega \in B(L^2(\GG))_*\}''$. In what follows, when we consider more than one quantum group, we will adorn the respective symbols with the upper index describing which group we refer to: so for example another (equivalent) way of defining $\hGG$ would be via the equality $W^{\hGG}= \sigma \left( (W^{\GG})^* \right)$. Note that $\Linf(\hGG)$ (and therefore $\C_0(\hGG)$) is naturally represented on $\Ltwo (\GG)$, and we have $W^{\GG} \in \M(\C_0(\hGG) \ot \C_0(\GG)) \subset \Linf(\hGG) \wot \Linf(\GG)$. In fact $W$ admits a universal version, $\WW \in \M(\C_0^u(\hGG) \ot \C_0^u(\GG))$, such that $W= (\Lambda_{\hGG} \ot \Lambda_{\GG})(\WW)$. We may also consider natural `one-sided' reduced versions, $\wW = (\Lambda_{\hGG} \ot \id)(\WW)$ and $\Ww= (\id \ot \Lambda_{\GG})(\WW)$. Occasionally we will also need the \emph{left multiplicative unitary}
$V \in \Linf (\GG) \wot \Linf(\hGG)'$, which implements the coproduct of $\GG$ via the formula
\[ \Delta(x) = V^* (\I \ot x)V, \;\;\; x \in \Linf(\GG).\]

We will also sometimes denote the objects related to $\hGG$ simply by using hats, so for example $\hh{\varphi}$ and $\hh{\psi}$ denote the left and right Haar weights of $\hGG$, respectively.

\subsection{Morphisms between quantum groups and closed quantum subgroups}

Given two locally compact quantum groups $\GG$ and $\HH$, a morphism from $\HH$ to $\GG$ is represented via a \cst-morphism $\pi\in \Mor(\C_0^u(\GG), \C_0^u(\HH))$ intertwining the respective coproducts:
\[ (\pi \ot \pi) \circ \Delta_{\GG} = \Delta_{\HH} \circ \pi.\]
It can be equivalently described via a \emph{bicharacter} from $\HH$ to $\GG$, i.e.\ a unitary $V \in \M (\C_0(\hh{\GG})\tens\C_0(\HH))$ such that
\[ (\Delta_{\hGG} \ot \id_{\C_0(\HH)}) (V) = V_{23} V_{13},\]
\[ (\id_{\C_0(\hGG)} \ot \Delta_{\HH}) (V) = V_{12} V_{13}.\]
In fact
$V= (\id \ot \Lambda^{\HH} \circ \pi)(\wW^{\GG})$. Each morphism $\pi$ from $\HH$ to $\GG$,
determines uniquely a \emph{dual morphism} $\hh{\pi}$ from $\hGG$ to $\hHH$
such that $(\hh{\pi} \ot \id)(\WW^{\GG})= (\id \ot  \pi)(\WW^{\HH})$. Finally note that although $\pi \in \Mor(\C_0^u(\GG), \C_0^u(\HH))$ describing a morphism from $\HH$ to $\GG$ need not have a reduced version $\pi_r \in \Mor(\C_0(\GG), \C_0(\HH))$ (such that $\pi_r \circ \Lambda^{\GG} = \Lambda^{\HH} \circ \pi$), if we are given  $\pi_r \in \Mor(\C_0(\GG), \C_0(\HH))$ intertwining the coproducts, then it always admits the universal version $\pi \in \Mor(\C_0^u(\GG), \C_0^u(\HH))$.
   For more information on this equivalence and other pictures of morphisms we refer to \cite{SLW12},\cite{DKSS}.

\begin{definition}
We say a morphism from $\HH$ to $\GG$ given by $\pi\in \Mor(\C_0^u(\GG), \C_0^u(\HH))$ identifies $\HH$ with a closed quantum subgroup of $\GG$ (in the sense of Vaes) if there exists an injective normal unital $^*$-homomorphism $\gamma:\Linf(\hHH) \to \Linf(\hGG)$ such that
\[\bigl.\gamma\bigr|_{\C_0(\hh{\HH})}\comp\Lambda_{\hh{\HH}}=\Lambda_{\hh{\GG}}\comp\hh{\pi}.\]
Often in this case we simply say $\HH$ is a closed quantum subgroup of $\GG$.
\end{definition}

The above definition is equivalent to the existence of an injective normal unital $^*$-homomor--phism $\gamma:\Linf(\hHH) \to \Linf(\hGG)$ intertwining the respective coproducts.
It then follows $\pi(\C_0^u(\GG))= \C_0^u(\HH)$ -- if the latter condition holds, we say that the underlying quantum group morphism identifies $\HH$ with a closed quantum subgroup of $\GG$ \emph{in the sense of Woronowicz}. These two notions are studied in detail in \cite{DKSS}, where in particular one can find the proofs of the facts stated above.

Finally note that it follows from \cite[Proposition 10.5]{BV}
that there is a bijective correspondence between closed quantum subgroups of $\GG$ and the so-called \emph{Baaj-Vaes} subalgebras of $\Linf(\hh{\GG})$ (i.e.\ those von Neumann subalgebras $\sN \subset \Linf(\hGG)$ for which $\Delta_{\hGG}(\sN)\subset\sN\wot\sN$, $\hh{R}(\sN)=\sN$ and $\hh{\tau}_t(\sN)=\sN$ for all $t\in\RR$).
More precisely, if $\sN$ is a Baaj-Vaes subalgebra of $\Linf(\hh{\GG})$ then there is a locally compact quantum group $\HH$ such that $\sN = \Linf(\HH)$, and more or less by definition $\hHH$ is a closed quantum subgroup of $\GG$.

We will later need the following simple lemma.
\begin{lemma} \label{lem:slicebicharacter}
Let $\HH$, $\GG$ be locally compact quantum groups and let $V \in \M (\C_0(\hh{\GG})\tens\C_0(\HH))$ be a bicharacter. Then the space
\[\{ (\id \ot \omega) V : \omega \in B(\Ltwo(\HH))_*\}''\]
is a Baaj-Vaes subalgebra of $\Linf(\hGG)$.
\end{lemma}
\begin{proof}
Considering $V$ as a unitary representation of $\HH$ on $\Ltwo(\GG)$ we conclude \[\sN = \{ (\id \ot \omega) V : \omega \in B(\Ltwo(\HH))_*\}''\]
is a von Neumann algebra.
Moreover, it follows from \cite[Proposition 3.15]{SLW12} that $\sN$ is preserved by $R^{\hGG}$ and $\tau^{\hGG}_t $ for all $t\in\mathbb{R}$. Finally, the bicharacter identity
\[(\Delta_{\hGG}\otimes\id)(V) = V_{23}V_{13}\] implies $\Delta_{\hGG}(\sN)\subset \sN\bar\otimes\sN$.
\end{proof}

\subsection{Actions of quantum groups and quantum  homogeneous spaces} \label{actions}

We use the notion of a (left) action of a quantum group on a von Neumann algebra in several occasions throughout the paper.
\begin{definition}
We say that a locally compact quantum group $\GG$ acts on a von Neumann algebra $\sM$ if there exists a unital injective normal $^*$-homomorphism $\alpha: \sM \to \Linf(\GG) \wot \sM$ such that
\[ (\Delta \ot \id_{\sM}) \comp \alpha =  (\id_{\Linf (\GG)} \ot  \alpha) \comp \alpha.\]
The map $\alpha$ is called a (left) action of $\GG$ on $\sM$.
\end{definition}

Each action as above automatically satisfies a von Neumann version of \emph{Podle\'s/nondegeneracy condition}: the set $\Lin \alpha(\sM) (\Linf(\GG) \wot \I_{\sM})$ is weak$^*$-dense in $\Linf(\GG) \wot \sM$ (\cite[Proposition 2.4]{KSprojections}).
The action $\alpha$ is said to be \emph{ergodic} if its \emph{fixed point algebra} $\textup{Fix}(\alpha):=\{x \in \sM: \alpha(x) = \I \ot x\}$ is trivial, i.e. equal to $\CC \I_{\sM}$.
A large source of ergodic actions is provided by the \emph{embeddable quantum homogeneous spaces} (\cite{KaspSol}), i.e.\ von Neumann subalgebras $\sN \subset \Linf(\GG)$ which are left coideals: $\Delta(\sN) \subset \Linf(\GG) \wot \sN$. The action of $\GG$ on $\sN$ is naturally given by the suitable restriction of the coproduct. In fact there is a natural co-duality between embeddable quantum homogeneous spaces for $\GG$ and for $\hGG$; if $\sN$ is a former, then $\tilde{\sN}=\{y \in \Linf(\hGG): \forall_{x \in \sN}\, xy=yx\}$ is a latter, and we have $\tilde{\tilde{\sN}}=\sN$ (\cite{KaspSol}).
A left coideal $\sL \subset \Linf(\GG)$ is said to be \emph{normal} if
\[
\ww^{\GG}(\I\tens\sL){\ww^{\GG}}^*\subset\Linf(\hh{\GG})\vtens\sL.
\]
All the notions and statements above have natural counterparts for the right actions (right coideals, etc.).

Assume that $\HH$ is a closed quantum subgroup of $\GG$, determined by a morphism $\pi \in \Mor (\C_0^u(\GG), \C_0^u(\HH))$. Then $\HH$ acts on $\Linf(\GG)$ on the right (we will modify the language slightly and say simply that $\HH$ acts on $\GG$) by the following formula
\begin{equation} \alpha_{\HH}(x) =  V(x \ot \I) V^*,\;\;\; x \in \Linf(\GG), \label{subgroupaction}\end{equation}
where $V\in \M(\C_0(\hGG) \ot \C_0(\HH))\subset \Linf(\hGG) \wot \Linf(\HH)$  denotes the bicharacter associated to the morphism $\pi$.

We then call the fixed point space of $\alpha_{\HH}$ the \emph{algebra of bounded functions on the quantum homogeneous space $\GG/\HH$} and denote it by $\Linf (\GG/\HH)$.

We will be also interested in its topological version, which should be a \cst-algebra contained in $\Linf(\GG / \HH)$, on which $\GG$ naturally acts (for the notion of action of a locally compact quantum group on a \cst-algebra we refer for example to \cite{SoltanActions} or \cite{Vaes-induction}). In general the problem of its existence remains open, but for regular quantum groups it was solved in \cite{Vaes-induction}, where the following result was shown.

\begin{theorem}\cite[Theorem 6.1]{Vaes-induction} \label{defhomspace}
Let $\GG$ be a regular locally compact quantum group and $\HH$ its closed quantum subgroup. Then there exists a unique \cst-algebra $\sA \subset \Linf (\GG/\HH)$ such that
\begin{rlist}
\item $\sA$ is dense in strong operator topology in $ \Linf (\GG/\HH)$;
\item $\Delta(\sA) \subset \M (\C_0 (\GG) \ot \sA)$ and  the map $\Delta|_{\sA}$ defines an action of $\GG$ on $\sA$;
\item $\Delta(\Linf(\GG/\HH)) \subset \M (\cK(\Ltwo (\GG)) \ot \sA)$ and  the restricted map $\Delta|_{\Linf(\GG/\HH)} \to \M (\cK(\Ltwo (\GG)) \ot \sA)$ is strictly continuous on bounded subsets.
\end{rlist}
We denote then $\sA$ by $\C_0(\GG/\HH)$ and call it the algebra of continuous functions on the quantum homogeneous space $\GG/\HH$.
\end{theorem}

\begin{remark}
The statements of Theorem 6.1 and Corollary 6.4 in \cite{Vaes-induction} formally involve the assumption that $\GG$ is \emph{strongly regular}. 
But, as already noted there (see the paragraph after Theorem 6.2 in that paper), regularity is sufficient for both results;
for the convenience of the reader we recall the details. 
Applying \cite[Theorem 6.7]{Vaes-induction} (which only requires regularity) to the case $B = \mathcal{K}(\mathcal{J})$,
where $\mathcal{J}$ is the \cst-$\C(\hHH)$-module defined in the beginning of the proof of 
\cite[Theorem 6.1]{Vaes-induction}, proves the existence of $\C_0(\GG/\HH)$ (and the uniqueness follows without any further assumptions as observed at the end of the same proof), as well as the isomorphism
$\mathcal{K}(\mathcal{J}) \cong \GG \ltimes \C_0(\GG/\HH)$, which then by the construction of
$\mathcal{J}$ implies the Morita equivalence $\C_0(\hHH) \sim \GG \ltimes \C_0(\GG/\HH)$.
Now the proof of \cite[Corollary 6.4]{Vaes-induction} only needs this fact together with the biduality theorem which again does not require any further assumptions on $\GG$. 
In fact, strong regularity is only needed to prove the imprimitivity theorem in the full crossed product setting.
\end{remark}

\subsection{Normal quantum subgroups and quotient quantum groups}

The definition of a closed normal quantum subgroup was introduced in \cite{VV}.
\begin{definition} \label{def:normal}
Let $\GG$ be a locally compact quantum group and $\KK$ its closed quantum subgroup identified by an injective morphism $\gamma:\Linf(\hKK) \to \Linf(\hGG)$. We say that $\KK$ is a normal quantum subgroup of $\GG$ if $\gamma(\Linf(\hKK))$ is a normal coideal in $\Linf(\hGG)$ (see the previous subsection).
\end{definition}

The key consequence (and actually a characterization) of the fact that $\KK$ is a normal quantum subgroup of $\GG$  is that  $\Linf(\GG/\KK)$ is a Baaj-Vaes subalgebra of $\Linf(\GG)$, so $  \GG/\KK$ becomes a locally compact quantum group, naturally called a \emph{quotient quantum group of $\GG$}. As $\Linf(  \GG/\KK)$ inherits all its quantum group structures (e.g. the antipode) from $\Linf(\GG)$, quotient quantum groups of quantum groups of Kac type are again of Kac type.

 The above facts lead naturally to the concept of short exact sequences, studied in detail in \cite{VV} (see also \cite{KSProjExt}). Denoting $\HH = \GG/\KK$ we have a short exact sequence
 \begin{equation}\label{ex1}\{e\}\rightarrow\KK\rightarrow\GG\rightarrow \HH\rightarrow\{e\}\end{equation}
together with the dual exact sequence
\begin{equation}\label{ex2}\{e\}\rightarrow\hat\HH\rightarrow\hat\GG\rightarrow\hat\KK\rightarrow\{e\}\end{equation}
where the embedding $\Linf(\HH)\subset \Linf(\GG)$,  yields the identification of   $\widehat{\HH}$ with a (Vaes) closed normal subgroup of $\hat\GG$.  Thus we have a bijective correspondence between normal subgroups of $\GG$ and $\hat\GG$ given by \eqref{ex1} and \eqref{ex2}.

Finally note that another characterisation of normality is based on the (alluded to earlier) notion of the \emph{left} homogeneous space $\KK \backslash \GG$: $\KK$ is normal in $\GG$ if and only if
\begin{equation}\label{normcondwang}\Linf(\GG/\KK)= \Linf(\KK \backslash \GG).
\end{equation}
In order to see this let us note that in general we have $R^\GG(\Linf(\GG/\KK)) = \Linf(\KK \backslash \GG)$. Since in the case of normal $\KK$, $\Linf(\GG/\KK)$ is a Baaj-Vaes subalgebra we conclude that   $\Linf(\GG/\KK)= \Linf(\KK \backslash \GG)$. Conversely, the   assumption \eqref{normcondwang}  shows that  $\Linf(\GG/\KK)$ is preserved by  $R^\GG$. Using \cite[Proposition 4.4]{KSprojections}  we conclude  that  $\KK\subset\GG$ is normal (for the compact case see \cite[Proposition 3.2]{Wangnormal} and \cite[Theorem 4.6]{KSSExact}).

\section{Definition of an open quantum subgroup}\label{Sec:opendef}

In this section we define open quantum subgroups and establish some of their basic properties.

As in the definitions of closed quantum subgroups, due to the fact that quantum groups are virtual objects, and not actual topological spaces, we need to reformulate the openness condition in the language of the associated `function algebras'. As usual, we do it first on the level of classical groups, to make sure that our definition coincides then with the straightforward topological concept.

There are in fact several equivalent ways to characterize openness of a subgroup of a locally compact group $G$.
Our definition is inspired by the following classical characterization of open subgroups due to Greenleaf.

\begin{theorem}\cite{Greenleaf65}\label{Greenleaf65}
Let $G$ and $K$ be locally compact groups.
$K$ is homeomorphic to an open subgroup of $G$ if and only if
there exists an injective homomorphism from $L^1(K)$ into $L^1(G)$.
\end{theorem}

Obviously one can dualize the above to find an equivalent characterization in terms of $L^\infty$-algebras, where it leads to the existence of a surjective map from $\Linf (G)$ onto $\Linf (K)$. This explains the following definition.

\begin{definition}\label{defm}
Let $\GG$ and $\HH$ be locally compact quantum groups. We say $\HH$ is an open quantum subgroup of $\GG$ if there is
a normal surjective unital $*$-homomorphism $\pi:\Linf(\GG)\rightarrow\Linf(\HH)$, intertwining the respective coproducts:
\[(\pi\tens\pi)\circ\Delta^{\GG} = \Delta^{\HH}\circ\pi.\]
\end{definition}

By Theorem \ref{Greenleaf65} open subgroups of locally compact groups $G$ are also open in the sense of Definition \ref{defm}
when $G$ is regarded as a locally compact quantum group.
It is also obvious from Definition \ref{defm} and for example \cite[Section 6]{DKSS} that every (closed) quantum subgroup of a discrete quantum group $\GG$ is an open quantum subgroup of $\GG$.

Definition \ref{defm} provides a connection between the (Haar) measure structures of $\GG$ and $\HH$, as expected for open subgroups.
Indeed we make this more precise later in this section by showing the compatibility of Haar weights and other related structures (also Theorem \ref{thm:weightsopen} will provide a converse statement: compatibility of Haar weights implies openness).

But first we show the map $\pi$ in Definition \ref{defm} induces a homomorphism from $\HH$ to $\GG$, i.e. a morphism on the \cst-algebraic level.

\begin{proposition}\label{vNC}
Suppose $\HH$ is an open quantum subgroup of $\GG$, identified via the map $\pi:\Linf(\GG)\rightarrow\Linf(\HH)$.
Then $\pi|_{\C_0(\GG)}\in\Mor(\C_0(\GG),\C_0(\HH))$. In particular $\pi|_{\C_0(\GG)}$ is  (the reduced version of) a morphism from $\HH$ to $\GG$.
\end{proposition}

\begin{proof}
Define a unitary $U = (\id\otimes\pi)\ww^\GG \in \Linf(\hat\GG)\overline\otimes \Linf(\HH)$.
Then
\begin{align*}
\ww^\HH_{23}U_{12} \ww^{\HH*}_{23}& = (\id\otimes\Delta^\HH)(U)
= (\id\otimes  (\pi\otimes\pi)\circ\Delta^\GG)(W^\GG)\\
& = (\id\otimes  \pi\otimes\pi)(W^\GG_{12}W^\GG_{13})
=  U_{12}U_{13} .
\end{align*}
Thus
\[
U_{13} = U_{12}^*\ww^\HH_{23}U_{12} \ww^{\HH*}_{23}\in\M(\cK(\Ltwo(\GG))\otimes\cK(\Ltwo(\HH))\otimes\C_0(\HH)),
\]
which implies
$U\in\M(\cK(\Ltwo(\GG)) \otimes\C_0(\HH))$, and therefore $\pi(\C_0(\GG))\subset\M(\C_0(\HH))$.

We have
\[\begin{split}
\pi(\C_0(\GG))\C_0(\HH)
&=
\{(\mu\otimes\nu\otimes\id) (U_{13}\ww^{\HH}_{23}) :  \mu\in\B(\Ltwo(\GG))_*, \, \nu\in \B(\Ltwo(\HH))_*\}^{-\|\cdot\|}
\\&=
\{(\mu\otimes\nu\otimes\id) (U_{12}^* \ww^{\HH*}_{23}U_{12}) :  \mu\in\B(\Ltwo(\GG))_*, \, \nu\in \B(\Ltwo(\HH))_*\}^{-\|\cdot\|}
\\&=
\C_0(\HH),
\end{split}\]
and similarly,
\[\begin{split}
\C_0(\HH) \pi(\C_0(\GG))
&=
{\{ (\mu\otimes\nu\otimes\id) (\ww^{\HH*}_{23}U_{13}^*):  \mu\in\B(\Ltwo(\GG))_*, \, \nu\in \B(\Ltwo(\HH))_*\}}^{-\|\cdot\|}
\\&=
\{(\mu\otimes\nu\otimes\id) (U_{12}^* \ww^{\HH}_{23}U_{12}) :  \mu\in\B(\Ltwo(\GG))_*, \, \nu\in \B(\Ltwo(\HH))_*\}^{-\|\cdot\|}
\\&= \C_0(\HH) .
\end{split}\]
Hence $\pi|_{\C_0(\GG)}\in\Mor(\C_0(\GG),\C_0(\HH))$.
\end{proof}

\begin{corollary}\label{antipodes-same}
Suppose $\pi:\Linf(\GG)\rightarrow\Linf(\HH)$ identifies $\HH$ with an open subgroup of $\GG$.
Then
$\pi\circ R^\GG = R^\HH\circ\pi$ and
$\pi\circ\tau^\GG_t =\tau_t^\HH\circ\pi$.

\end{corollary}

\begin{proof}
Both equalities follow from Proposition \ref{vNC} and \cite[Proposition 5.45]{KV}.
\end{proof}

The following lemma will be further strengthened (under certain technical assumptions) in Section \ref{Sec:BKLS}.

\begin{lemma}\label{emded}
Suppose $\pi$ identifies $\HH$ with an open subgroup of $\GG$. Consider the (universal) $\C^*$-version $\pi_u\in\Mor(\C_0^u(\GG), \C_0^u(\HH))$ and let
$\hat\pi\in\Mor(\C_0^u(\hat\HH),\C_0^u(\hat\GG))$ be the dual morphism.
Then $\hat\pi(\C_0^u(\hat\HH))\subseteq\C_0^u(\hat\GG)$.
\end{lemma}
\begin{proof}
Denote by $\pi_*:\Lone(\HH)\rightarrow\Lone(\GG)$ the preadjoint of $\pi$. Then
\[
\hat\pi((\id\otimes\omega) \Ww^\HH) = (\id\otimes\pi_*(\omega)) \Ww^\GG \in \C_0^u(\hat\GG)
\]
for all $\omega\in\Lone(\HH)$. Since
\[
\{(\id\otimes\omega) \Ww^\HH : \omega\in\Lone(\HH)\}^{-\|\cdot\|} = \C_0^u(\hat\HH)
\]
the assertion follows.
\end{proof}

In the classical setting the fact that an open subgroup is also closed provides a well-behaved projection, namely the characteristic
function of the subgroup. This plays an essential role in the analysis of the subgroup and its properties.

In our context it does not follow immediately from the definition that open quantum subgroups are also closed (see Section \ref{open->closed}), but we still obtain a well-behaved projection which will be useful in the analysis to follow.

\begin{lemma}\label{charac-func}
Suppose $\pi$ identifies $\HH$ with an open subgroup of $\GG$. Let $P\in Z(\Linf(\GG))$ be the central support of $\pi$. Then
\begin{enumerate}
 \item $\Delta^{\GG}(P)(P\otimes P) = P\otimes P$;
 \item $R^{\GG}(P) = P$;
 \item $\tau^{\GG}_t(P) = P$.
\end{enumerate}
\end{lemma}

\begin{proof}
Let $\iota:\Linf(\HH) \to P \Linf(\GG)$ be an isomorphism given by the formula
\[ \iota (\pi(x)) = Px, \;\;\; x \in \Linf(\GG).\]
For any $x \in \Linf(\GG)$
\[ (\iota \ot \iota) \left( \Delta_{\HH}(\pi(x)) \right) =  (\iota \ot \iota) \left( (\pi \ot \pi)(\Delta_{\GG} (x)) \right) = (P \ot P) \Delta_{\GG} (x).\]
We obviously have $P= \iota(\I)$, so further
\[ (P \ot P) (\Delta_{\HH} (P)) = (\iota \ot \iota) \left( \Delta_{\HH}(\pi(\I)) \right) = P \ot P.\]
Using Corollary \ref{antipodes-same}, as $\pi\circ R^\GG=R^\HH\circ \pi$,
we get
\[ P R^{\GG} (P) = \iota (\pi (R^\GG(P))) =  \iota(R^{\HH} (\pi(P)) = \iota (R^{\HH}(\I))= \iota (\I) = P.\]
Applying $R^\GG$ to the above identity we obtain
\[
R^\GG(P) = P R^\GG(P),
\]
and (2) follows.

To obtain (3), again using Corollary \ref{antipodes-same} we first observe that
\[
P = \tau_t^\HH(\pi(P)) = \pi (\tau_t^\GG(P)) = P \tau_t^\GG(P)
\]
for all $t\in\mathbb R$, and then apply $\tau_{-t}^\GG$ to this equality to obtain
\[
\tau_{-t}^\GG(P) = P \tau_{-t}^\GG(P)
\]
for all $t\in\mathbb R$, which implies (3).
\end{proof}

From now on we shall simply write $\mathds{1}_\HH$ for the projection $P$ of the above proposition (calling it the support of $\HH$), and
freely use the identification $\Linf(\HH)\cong \mathds{1}_\HH\Linf(\GG)$ whenever convenient, not mentioning explicitly the identifying isomorphism $\iota$. In the next lemma we note that this picture is also compatible with the respective coproducts.

\begin{proposition}\label{delta(P)}

Let $\GG$ be a locally compact quantum group and $\HH$ an open quantum subgroup of $\GG$. Then
\[
\Delta^{\GG}(\mathds{1}_\HH)(\I\otimes \mathds{1}_\HH) = \mathds{1}_\HH\otimes \mathds{1}_\HH=\Delta^{\GG}(\mathds{1}_\HH)(\mathds{1}_\HH\otimes \I) .
\]
\end{proposition}
\begin{proof}
First note that it is enough to prove $\Delta^{\GG}(\I_\HH)(\I\otimes \I_\HH) = \I_\HH\otimes \I_\HH$. The second equality follows then from the identity $R^{\GG}(\I_\HH) = \I_\HH$
which was proved in Lemma \ref{charac-func}.

By Lemma \ref{charac-func} we have (writing $\I$ for the unit of $\Linf(\GG)$)
\[\begin{split}
\I_\HH\otimes \I_\HH &= \Delta^{\GG}(\I_\HH)(\I_\HH\otimes \I_\HH)
\\&=
W^\GG(\I_\HH\otimes \I)W^{\GG*}(\I_\HH\otimes \I_\HH)
\\&=
W^{\GG}(\I_\HH\otimes \I_\HH)W^{\GG*}(\I_\HH\otimes \I_\HH)
\end{split}\]
which implies
\begin{equation}\label{123}
\I_\HH\otimes \I_\HH \leq W^\GG(\I_\HH\otimes \I_\HH)W^{\GG*} .
\end{equation}

On the other hand, since
$\I_\HH$ is central and $R^\GG(\I_\HH) = \I_\HH$ we have
\[
J^\psi\I_\HH J^\psi = \I_\HH =  J^{\hat\psi} \I_\HH  J^{\hat\psi}.
\]
Therefore, if we multiply by $ J^{\psi}\otimes  J^{\hat\psi}$ to both sides of the inequality \eqref{123} we obtain
\[\begin{split}
\I_\HH\otimes \I_\HH &= ( J^{\psi}\otimes  J^{\hat\psi}) (\I_\HH\otimes \I_\HH) ( J^{\psi}\otimes  J^{\hat\psi})
\\&\leq
( J^{\psi}\otimes  J^{\hat\psi})W^\GG(\I_\HH\otimes \I_\HH)W^{\GG*}( J^{\psi}\otimes  J^{\hat\psi})
\\&=
W^{\GG*}(\I_\HH\otimes \I_\HH)W^\GG .
\end{split}\]
Hence
\[\begin{split}
\I_\HH\otimes \I_\HH &= W^\GG(\I_\HH\otimes \I_\HH)W^{\GG*}
\\&= W^\GG(\I_\HH\otimes \I)W^{\GG*} (\I\otimes \I_\HH)
\\&=
\Delta^\GG(\I_\HH)(\I\otimes \I_\HH),
\end{split}\]
as desired.
\end{proof}

We are now ready to see that if $\HH$ is an open quantum subgroup then the inclusion $\HH \subset \GG$ respects also the Haar weights.

\begin{theorem}\label{lemh}
Let $\GG$ be a locally compact quantum group and $\HH \subset \GG$ an open quantum subgroup.
Then the left and right Haar weights on $\HH$ are the restrictions of the Haar weights of $\GG$ (so that in particular the corresponding modular automorphism groups are also compatible).
Moreover $\delta^\HH = \I_\HH\delta^\GG$.
\end{theorem}
\begin{proof}
First note that since $\I_\HH$ is central, the restriction of $\varphi^\GG$ to  $\Linf(\HH) = \I_\HH \Linf(\GG)$ defines an n.s.f.\  weight on $\Linf(\HH)$.
Moreover, using Proposition \ref{delta(P)} we have for every $x\in \Linf(\HH)_+$

\[\begin{split}
(\id\otimes \varphi^\GG) \Delta^\HH(x)
&=
(\id\otimes \varphi^\GG)(\Delta^\GG(x)(\I_\HH\otimes \I))
=
\I_\HH\varphi^\GG(x)
\end{split}\]
which shows the restriction of $\varphi^\GG$ to $\Linf(\HH)$ is left invariant (with respect to comultiplication of $\HH$).
The case of the right Haar weight is similar.
Hence by the uniqueness of Haar weights the first assertion follows.

Now since $\psi^\GG = \varphi^\GG_{\delta^\GG}$ and $\psi^\HH = \varphi^\HH_{\delta^\HH}$ we conclude that $\delta^\HH = \I_\HH\delta^\GG$ (see \cite[Proposition 5.1]{VJOT}).
\end{proof}

We can now easily deduce the following corollary.

\begin{corollary} \label{cor:unitaryHG}
Let $\GG$ be a locally compact quantum group and $\HH \subset \GG$ an open quantum subgroup. Then we have a natural embedding $\Ltwo(\HH) \subset \Ltwo(\GG)$ and $\ww^{\HH} = (\I_{\HH} \ot \I_{\HH}) \ww^{\GG}|_{\Ltwo(\HH) \ot \Ltwo (\HH)}$.
\end{corollary}

\section{Open quantum subgroups are closed}\label{open->closed}

In this section we prove that an open quantum subgroup in the sense of Definition \ref{defm} is also closed in the sense of Vaes.
In contrast to the classical case this turns to be rather non-trivial.
Our first task is to find a more concrete description of the homogeneous spaces in the case of open quantum subgroups.

Throughout this section $\GG$ is a locally compact quantum group and $\HH$ is an open quantum subgroup of $\GG$ with the corresponding map $\pi: \Linf(\GG) \to \Linf(\HH)$.

Define
\begin{equation}\label{HactG}
\alpha:\Linf(\GG)\ni x\mapsto(\id\otimes\pi) \Delta^\GG(x) \in \Linf(\GG)\bar\otimes\Linf(\HH)
\end{equation}
and let
\[
\sZ = \{x\in\Linf(\GG):\alpha(x) = x\otimes\I\}.
\]
Note that we are using a different notation than the one introduced earlier in Subsection \ref{actions}, as, formally speaking,
we do not know yet whether $\HH$ is a closed subgroup and how the action introduced above is connected to the one discussed in Subsection \ref{actions}.
All this will turn out to be compatible later on.

Observe that $\I_\HH\in \sZ$ and
$\Delta^\GG(x)\in\Linf(\GG)\bar\otimes \sZ$ for all $x\in\sZ$. The restriction of $\Delta^\GG$ to $\sZ$ will be denoted by $\beta$:
\[
\beta:\sZ\rightarrow\Linf(\GG)\bar\otimes \sZ .
\]
Clearly $\beta$ is an ergodic action of $\GG$ on $\sZ$.
Define the unitary $U\in\Linf(\hat\GG)\bar\otimes\Linf(\HH)$ by the formula
\begin{equation}\label{defU}
U = (\id\otimes\pi) \ww^\GG.
\end{equation}
Then
\begin{equation}\label{cod}
x\in\sZ\iff U(x\otimes\I) = (x\otimes\I)U .
\end{equation}

The algebra $\sZ$ satisfies certain natural invariance properties.

\begin{lemma}\label{quot0}
We have (for each $t \in \RR$)
\begin{enumerate}
\item\label{siginv}
$\tau^{\GG}_t(\sZ) = \sZ$,
\item\label{tauinv}
$\sigma_t^{\varphi}(\sZ) = \sZ$.
\end{enumerate}
\end{lemma}

\begin{proof}
Using Corollary \ref{antipodes-same}, we obtain
\[
\alpha\circ\tau^\GG_t = (\id\otimes\pi) \Delta^\GG\circ \tau^\GG_t
= (\id\otimes\pi) (\tau^\GG_t\otimes \tau^\GG_t) \Delta^\GG
= (\tau^\GG_t\otimes \tau^\HH_t) (\id\otimes\pi) \Delta^\GG
= (\tau_t^\GG\otimes\tau_t^\HH)\circ\alpha .
\]
Hence, for $x\in \sZ$ it follows that
\[
\alpha(\tau^\GG_t (x)) = (\tau_t^\GG\otimes\tau_t^\HH)\alpha (x) = \tau_t^\GG(x) \otimes\I,
\]
which yields (1).

For (2), first note that Theorem \ref{lemh} implies $\sigma^{\varphi^{\HH}}_t\circ\pi = \pi\circ\sigma^{\varphi}_t$. So, applying \eqref{delta(sigma)} we obtain
\[
\alpha\circ\sigma^{\varphi}_t = (\tau_t^\GG\otimes \sigma^{\varphi^{\HH}}_{t})\circ\alpha .
\]
Hence, for $x\in\sZ$ we get
\[
\alpha(\sigma^{\varphi}_t(x)) = \tau^\GG_t(x)\otimes \I,
\]
or equivalently (see \eqref{defU})
\[
U(\sigma^{\varphi}_t(x)\otimes \I) = (\tau^\GG_t(x)\otimes \I)U,
\]
which gives
\[
((\id\otimes\omega)U) \sigma^{\varphi}_t(x) = \tau^\GG_t(x) ((\id\otimes\omega)U)
\]
for all $\omega\in\B(\Ltwo(\HH))_*$. But the fact that
\[
\I\in\{(\id\otimes\omega)U:\omega\in\B(\Ltwo(\HH))_*\}''
\]
implies
\begin{equation}\label{mods-coincide}
\sigma^{\varphi}_t(x) =\tau^\GG_t(x) .
\end{equation}
Therefore (2) follows from (1).
\end{proof}

Consider the case of a locally compact group $G$ and an open subgroup $H\leq G$. Then it is easy to see that the action $\alpha$ describes the standard action of $G$ on $H$ and we have
\[
\ell^\infty(G/H) = \{t\delta_H : t\in G\}'',
\]
where $\delta_H$ is the Dirac function at point $[H]\in G/H$, and $t\delta_H$ denotes the canonical action of the element $t\in G$ on the function $\delta_H\in L^\infty(G)$.
Our next theorem leads to the quantum version of the latter description of the homogeneous space.

But first let us show that $\I_{\HH}$ is in fact the quantum counterpart of $\delta_H$.

 \begin{proposition}\label{1_H:counit}
Let $\HH\subset \GG$ be an open quantum subgroup of a locally compact quantum group  $\GG$  and $\I_{\HH}$ the support of $\HH$.
Then $\I_\HH$ is a minimal central projection of $\sZ$.
Moreover, if $\omega$ is a normal state on $\sZ$ with $\omega(\I_\HH) =1$ then we have $(\eta \otimes \I_\HH \omega) \Delta^\GG(x) = \eta(x)$
for all $x\in \sZ$ and $\eta\in \sZ_*$.

 \end{proposition}
\begin{proof}
We know $\I_\HH$ is central. To show minimality, suppose $q\in \sZ$ and $q\leq \I_{\HH}$. Then $q = q \I_{\HH} \in L^\infty(\HH)$,
and since moreover $q\in \sZ$ it follows $\Delta^\HH(q) = (\I\otimes \I_\HH) \Delta^\GG(q) = q \otimes \I_\HH$.
Hence $q = 0$ or $q = \I_\HH$ by \cite[Lemma 6.4]{KV}.

For the last part of the statement, we observe
\[
(\eta \otimes \I_\HH \omega) \Delta^\GG(x) = (\eta \otimes \omega) (\I \otimes \I_\HH) \Delta^\GG(x) = (\eta \otimes \omega) (x \otimes \I_\HH) = \eta(x) .
\]
\end{proof}

We are now ready to provide an alternative description of the algebra $\sZ$.

\begin{theorem}\label{quot}
We have
\[
\sZ  = \overline{\{(\omega\otimes\id) \Delta^\GG(\I_\HH):\omega\in\B(\Ltwo(\GG))_*\}}^{\textrm{weak}}.
\]
\end{theorem}

\begin{proof}
For every $\omega\in\B(\Ltwo(\HH))_*$ we have
\begin{align*}
\alpha((\omega\otimes\id)\Delta^\GG(\I_\HH))  &= (\omega\otimes\id\otimes\id)(\id\otimes\alpha)(\Delta^\GG(\I_\HH))= (\omega\otimes\id\otimes\id)(\Delta^\GG\otimes \id)(\alpha(\I_\HH))\\
&= (\omega\otimes\id\otimes\id)(\Delta^\GG\otimes \id)(\I_\HH\otimes\I) = ((\omega\otimes\id) \Delta^\GG(\I_\HH))\otimes \I ,
\end{align*}
which shows $(\omega\otimes\id) \Delta^\GG(\I_\HH)\in\sZ$.

Using the strong invariance of the antipode, for $a, b \in \mathcal{N}_{\varphi}$ and $x\in\sZ\cap\D(\sigma^{\varphi}_{\frac{i}{2}})\cap \D(S^\GG)$
we obtain
\begin{align*}
(\id\otimes\varphi)((\I\otimes a^*)&\Delta^\GG(b)\Delta^\GG(\I_\HH)(S^\GG(x)\otimes\I))
 = S^\GG((\id\otimes\varphi)((x\otimes \I_\HH)\Delta^\GG(   a^*)(\I\otimes \I_\HH b)))\\
& = S^\GG((\id\otimes\varphi)(\Delta^\GG( x a^*)(\I\otimes \I_\HH b)))
 = (\id\otimes\varphi)( (\I\otimes x a^*)\Delta^\GG(b)\Delta^\GG(\I_\HH))
\\ &=  (\id\otimes\varphi)( (\I\otimes  a^*)\Delta^\GG(b)\Delta^\GG(\I_\HH)(\I\otimes \sigma^{\varphi}_{\frac{i}{2}}(x))),
\end{align*}
which implies, by faithfulness of $\varphi$ and `density conditions' for $\GG$,
\begin{equation}\label{charcomm0}
\Delta^\GG(\I_\HH)(S^\GG(x)\otimes\I) = \Delta^\GG(\I_\HH)(\I\otimes \sigma^{\varphi}_{\frac{i}{2}}(x))
\end{equation}
for all $x\in\sZ\cap\D(\sigma^{\varphi}_{\frac{i}{2}})\cap \D(S^\GG)$.
But $\D(S^\GG) = \D(\tau^{\GG}_{\frac{-i}{2}})$, and also by Lemma \ref{quot0} we have $\sZ\cap\D(\tau^{\GG}_{\frac{-i}{2}}) = \sZ\cap\D(\sigma^{\varphi}_{\frac{-i}{2}})$ .
Thus, by replacing $x$ with $\sigma^{\GG}_{-\frac{i}{2}}(x)$ in \eqref{charcomm0} we obtain
\begin{equation}\label{charcomm}
\Delta^\GG(\I_\HH)(\I\otimes x) = \Delta^\GG(\I_\HH)(R^\GG(x)\otimes\I)
\end{equation}
for all $x\in\sZ\cap\D(\sigma^{\varphi}_{\frac{i}{2}})\cap \D(S^\GG)$.

It follows from Theorem \ref{lemh} that $\sZ\cap\D(\sigma^{\GG}_{\frac{i}{2}})\cap \D(\sigma^{\GG}_{\frac{-i}{2}})$
is weak$^*$ dense in $\sZ$.
Hence, we conclude that \eqref{charcomm} holds for all $x\in\sZ$.
In particular, for $\omega\in\B(\Ltwo(\HH))_*$ and $x\in\sZ$ we have
\[\begin{split}
\big((\omega\otimes\id)\Delta^\GG(\I_\HH)\big) x &= (\omega\otimes\id) (\Delta^\GG(\I_\HH) (\I\otimes x)) =
(\omega\otimes\id) (\Delta^\GG(\I_\HH)(R^\GG(x)\otimes\I))
\\&= (R^\GG(x)\omega \otimes\id) \Delta^\GG(\I_\HH),
\end{split}\]
which implies that the von Neumann subalgebra
\[
\sN := \overline{\{(\omega\otimes\id) \Delta^\GG(\I_\HH):\omega\in\B(\Ltwo(\GG))_*\}}^{\textrm{weak}} \subseteq \sZ
\]
is in fact a (two-sided) ideal in $\sZ$.

We next show $\sN$ is $\GG$-invariant (i.e.\ $\Delta(\sN) \subset \Linf(\GG) \wot \sN$).
For this, let $x,y\in\Ltwo(\GG)$ and let $(e_i)_{i \in \Jnd}$ be an orthonormal basis of $\Ltwo(\GG)$. Then
\[\begin{split}
\Delta^\GG((\omega_{x,y}\otimes\id)(\Delta^\GG(\I_\HH))) &= (\omega_{x,y}\otimes\id\otimes\id)((\Delta^\GG\otimes\id)\circ\Delta^\GG(\I_\HH))\\
&= (\omega_{x,y}\otimes\id\otimes\id)(W_{12}W_{13}(\I_\HH\otimes\I\otimes\I)W_{13}^*W_{12}^*)\\
&= \sum_{i,j \in \Jnd}(\omega_{x,e_i}\otimes\id)(W)(\omega_{e_j,y}\otimes\id)(W^*)\otimes (\omega_{e_i,e_j}\otimes\id)(\Delta^\GG(\I_\HH)),
\end{split}\]
which implies $\sN$ is $\GG$-invariant.

Since $\sN$ is an ideal in $\sZ$, there is a central projection $q\in \sZ$ such that $\sN = q \sZ$. The invariance of $\sN$ implies $\Delta^\GG(q) \leq \I\otimes q$.
Then it follows from \cite[Lemma 6.4]{KV} that $q=\I$, hence $\sN = \sZ$.
\end{proof}

In the next lemma we identify the co-dual coideal in $\hGG$ related to $\sZ$.

\begin{lemma}\label{M_1=M_2}
We have
\[
\{(\id\otimes\omega) (\id\otimes\pi) \ww^\GG :\omega\in\B(\Ltwo(\HH))_*\}'' = \{x\in\Linf(\hat\GG): x\I_\HH = \I_\HH x\} .
\]

\end{lemma}

\begin{proof}
Let $\sM_1:= \{(\id\otimes\omega) (\id\otimes\pi) \ww^\GG :\omega\in\B(\Ltwo(\HH))_*\}''$
and $\sM_2:= \{x\in\Linf(\hat\GG): x\I_\HH = \I_\HH x\}$.
Since $(\id\otimes\pi) \ww^\GG$ is a bicharacter, $\sM_1$ satisfies the Baaj-Vaes conditions by Lemma \ref{lem:slicebicharacter}.
Moreover, thanks to \eqref{cod}, $\sM_1$ may be viewed as the codual coideal of $\sZ$ (see Subsection \ref{actions}).
Thus by the results of \cite{KaspSol}
\begin{equation}\label{mo1}
\sM_1 = \{y\in \Linf(\hat\GG):xy = yx \textrm{ for all } y\in\sZ\},
\end{equation}
and since $\I_\HH \in\sZ$, we get $\sM_1\subseteq \sM_2$.

For the reverse inclusion,
note that using the left multiplicative unitary $V^{\GG}$ we have
\[
(\I\otimes x) \Delta^\GG(\I_\HH) = (V^\GG)^* (\I\otimes x\I_\HH) V^\GG = (V^\GG)^* (\I\otimes \I_\HH x) V^\GG =
\Delta^\GG(\I_\HH)(\I\otimes x)
\]
for all $x\in \sM_2$.
Therefore we get
\[
x \,((\omega\otimes\id) \Delta^\GG(\I_\HH))  = ((\omega\otimes\id) \Delta^\GG(\I_\HH))\, x
\]
for all $\omega\in\B(\Ltwo(\HH))_*$. Combining the equality \eqref{mo1} with Theorem \ref{quot} we get $\sM_2\subseteq \sM_1$.
\end{proof}

We need one more technical lemma before proving the main result of this section.
We believe it is well-known to experts, but we include the proof for completeness as we could not find a
reference. 
\begin{lemma}\label{prod}
Let $\GG$ be a locally compact quantum group and $x\in \Linf(\hGG)$, $y\in\Linf(\GG)$. If $xy = 0$ then $x=0$ or $y=0$.
\end{lemma}

\begin{proof}
Let $V$ denote the left multiplicative unitary of $\GG$. If $x\in \Linf(\hGG)$, $y\in\Linf(\GG)$, and $xy = 0$, then
\[
  V^*(\I\otimes xy) V = (\I\otimes x)V^*(\I\otimes y) V
\]
and we get $(\I\otimes x)\Delta^\GG(y) = (\I \otimes x)\ww^\GG(y\otimes\I)\ww^{\GG*}= 0$. In particular
\[(\I\otimes x)\ww^\GG(y\otimes\I) = 0\]
and we conclude that
\[(\I\otimes x)(\Linf(\hat\GG)\otimes\I)\ww^\GG(\I\otimes\Linf(\GG))(y\otimes\I) = 0.\]
Let us then consider the (right) action of $\GG$ on $\Linf(\hat\GG)$ given by the map
\[\delta: \Linf(\hat\GG)\ni y\mapsto \ww^{\GG*} (y\otimes \I)\ww^{\GG}\in\Linf(\hat\GG)\bar\otimes\Linf(\GG) \]
The von Neumann version of the Podle\'s condition for $\delta$ implies that $\Lin((\Linf(\hat\GG)\otimes\I)\ww^\GG(\I\otimes\Linf(\GG))(\ww^{\GG*})$ is weak$^*$-dense in  $\Linf(\hat\GG)\bar\otimes\Linf(\GG)$,
and by unitarity of $\ww^{\GG}$, so is
  \[(\Linf(\hat\GG)\otimes\I)\ww^\GG(\I\otimes\Linf(\GG)) = \Linf(\hat\GG)\bar\otimes\Linf(\GG). \]
This yields in turn
\[( \I\otimes x)(\Linf(\hat\GG)\bar\otimes\Linf(\GG))(y\otimes \I) = 0\]  In particular $x\otimes y =0$ and we conclude that $x=0$ or $y=0$.
\end{proof}

\begin{theorem}\label{Vaescl}
Suppose $\GG$ is a locally compact quantum group and $\HH \leq \GG$ is an open quantum subgroup. 
Then $\HH$ is a closed quantum subgroup of $\GG$ in the sense of Vaes.
\end{theorem}
\begin{proof}
In view of Theorem \ref{lemh} we may identify $\Ltwo(\HH)$ with $\I_\HH\Ltwo(\GG)$. Then under this identification,
Lemma \ref{M_1=M_2} gives
\[
\ww^\HH = (\I_\HH\otimes \I_\HH)\ww^\GG (\I_\HH\otimes \I_\HH) = (\I_\HH\otimes \I_\HH)\ww^\GG = \ww^\GG (\I_\HH\otimes \I_\HH) .
\]
Applying Lemma \ref{M_1=M_2} once more we get
\[\begin{split}
\Linf(\hat\HH) &= \{(\id\otimes \omega) \ww^\HH : \omega\in\B(\Ltwo(\HH))_*\}'' \\&=
\{(\id\otimes \omega) (\I_\HH\otimes \I_\HH)\ww^\GG (\I_\HH\otimes \I_\HH): \omega\in\B(\Ltwo(\HH))_*\}''
\\&=
\{\I_\HH \,\big((\id\otimes\omega) (\id\otimes\pi) \ww^\GG\big) \,\I_\HH:\omega\in\B(\Ltwo(\HH))_*\}''
\\&=
\I_\HH \,\{x\in\Linf(\hat\GG): x\I_\HH = \I_\HH x\} .
\end{split}\]

Hence, the mapping $x\mapsto x\I_\HH$ gives a normal *-homomorphism from the von Neumann subalgebra
$\{x\in\Linf(\hat\GG): x\I_\HH = \I_\HH x\}$ of $\Linf(\hat\GG)$ onto $\Linf(\hat\HH)$, which is faithful by Lemma \ref{prod}.

We denote by $\gamma: \Linf(\hat\HH) \to \Linf(\hat\GG)$, $x\I_\HH\mapsto x$, the inverse of the above map. Then
\[\begin{split}
(\gamma\otimes\gamma) \Delta^{\hat\HH}(x\I_\HH) &= (\gamma\otimes\gamma) (\ww^\HH)^* (\I\otimes x\I_\HH) \ww^\HH
\\&=
(\gamma\otimes\gamma) (\I_\HH\otimes \I_\HH) (\ww^\GG)^* (\I_\HH\otimes \I_\HH) (\I\otimes x\I_\HH) (\I_\HH\otimes \I_\HH)\ww^\GG (\I_\HH\otimes \I_\HH)
\\&=
(\gamma\otimes\gamma) (\I_\HH\otimes \I_\HH) (\ww^\GG)^* (\I\otimes x) \ww^\GG=
\Delta^{\hat\GG}(\gamma(x\I_\HH)),
\end{split}\]
which implies $\HH$ is a Vaes closed subgroup of $\GG$.
\end{proof}

\begin{remark}\label{redinj}
Note from the proof above we can also deduce that
\[
\gamma \left((\id\otimes \omega) \ww^\HH\right) = \gamma \left(\I_\HH \,\big((\id\otimes\omega) (\id\otimes\pi) \ww^\GG\big)\right)
= (\id\otimes\omega) (\id\otimes\pi) \ww^\GG
\]
for all $\omega\in\B(\Ltwo(\HH))_*$.
This in particular implies that
\[
\gamma(\C_0(\hat\HH))\subset\C_0(\hat\GG)
\]
(note that we have $\gamma \in \Mor(\C_0(\hHH), \C_0(\hGG))$).
\end{remark}

We record one more consequence of the above theorem, following essentially from the fact that closed quantum subgroups in the sense of Vaes are automatically closed in the sense of Woronowicz.

\begin{corollary} \label{cor:surjecC0}
Let $\HH$ be an open quantum subgroup of a locally compact quantum group $\GG$ identified via a surjective morphism $\pi:\Linf(\GG) \to \Linf(\HH)$. Then $\pi(\C_0(\GG))= \C_0(\HH)$.
\end{corollary}

Finally we see that the objects defined in the beginning of the section are indeed the familiar ones.

\begin{corollary} \label{standardhomog}
The map $\alpha$ defined in \eqref{HactG} is the canonical right action of $\GG$ on $\HH$ (see Subsection \ref{actions}). In particular we have $\sZ=\Linf(\GG/\HH)$.
\end{corollary}
\begin{proof}
This follows from two observations: first, if $\gamma: \Linf(\hat\HH) \to \Linf(\hat\GG)$ is the morphism introduced in the proof of Theorem \ref{Vaescl}, and say $\rho:\C_0^u(\HH) \to \C_0^u(\GG)$ its universal lift, then the map $\pi:\Linf(\GG) \to \Linf(\HH)$ identifying $\HH$ as an open subgroup of $\GG$ extends the reduced version of $\hat{\rho} \in \Mor(\C_0^u(\GG), \C_0^u(\HH))$ -- this effectively is a compatibility of two points of view on $\HH$ as a subgroup of $\GG$. The easiest way to see this is noting that we have the following equality:
\[ (\gamma \ot \id) (\ww^{\GG})=(\id \ot \pi)(\ww^{\HH}),\]
which in turn follows from Corollary \ref{cor:unitaryHG}.

Second, in view of the formula \eqref{subgroupaction} the canonical action of $\GG$ on $\HH$ is given by the formula
\begin{align*}
\beta(x) &= V(x \ot \I) V^* = (\id \ot \Lambda^{\HH} \circ \hat{\rho})(\wW^{\GG}) (x \ot \I)(\id \ot \Lambda^{\HH} \circ \hat{\rho})(\wW^{\GG})^*
\\&= (\id \ot \pi)(\ww^{\GG}) (x \ot \I)(\id \ot \pi)(\ww^{\GG})^* = (\id \ot \pi) ((\ww^{\GG}) (x \ot \I))(\ww^{\GG})^*) \\&= (\id \ot \pi)(\Delta_{\GG}(x)
\end{align*}
for $x \in \Linf(\GG)$, which shows it indeed coincides with the one defined by \eqref{HactG}.
\end{proof}

The following result will be used in Section 5.

\begin{theorem}\label{minimal}
Suppose that $\mathbb{H}\subset \mathbb{G}$ is a closed  quantum subgroup  and that morphism
 $\pi\in\M(\C_0^u(\GG),\C^u_0(\HH))$ admits a reduced version $\pi_r\in\M(\C_0(\GG),\C_0(\HH))$. Assume further that $\Linf (\mathbb{G}/\mathbb{H})$  admits a  minimal  projection $P\in Z(\Linf(\mathbb{G}/\mathbb{H}))$ such that $P \in  \M(\C_0(\mathbb{G}))$ and $\pi_r(P) = 1$. Then $\mathbb{H}$ is open in $\mathbb{G}$.
\end{theorem}
\begin{proof}
As the space $\Linf(\mathbb{G}/\mathbb{H})$ is a normal coideal,  $\ww(1\otimes P)\ww^*\in \Linf(\hat{\mathbb{G}})\bar\otimes \Linf(\mathbb{G}/\mathbb{H})$. Since $P$ is minimal,  there exists $x\in \Linf(\hat{\mathbb{G}})$ such that \begin{equation}\label{nor}\ww(1\otimes P)\ww^*(1\otimes P) = x\otimes P.\end{equation}
View the above equality (and the equalities to follow) as an equality of elements of $\M(\K(\Ltwo(\GG)) \ot \C_0(\GG))$. Applying the map $\id\otimes \pi_r$ to both sides we conclude that $x =1$ and obtain the equality
\[(1\otimes P)\ww^* (1\otimes P) =  \ww^*(1\otimes P).\]
In turn applying to both sides the maps $\omega\otimes\id$ for all  $\omega\in\B(\Ltwo(\GG))_*$ we see that in fact $PaP = aP$ for all $a\in\Linf(\GG)$. In particular $a^*P = Pa^*P = (PaP)^* = (aP)^* =  Pa^*$, i.e.\ $P\in Z(\Linf(\mathbb{G}))$.

Using the minimality of $P$ once again we see that  for every $z\in\Linf(\GG/\HH)$ there exists $y_z\in \Linf(\mathbb{G})$ such that \[\Delta(z)(1\otimes P) = y_z\otimes P.\] Applying $\id\otimes \pi_r$ to the above shows finally that $y_z = z$, i.e.
\begin{equation}\label{grl}\Delta(z)(1\otimes P) = z\otimes P.\end{equation}
In particular, taking  $z = P$ we conclude that $P$ is a group like projection. Let $\mathbb{H}'$  be the corresponding  open quantum subgroup of $\GG$.  Equality \eqref{grl} shows  that $\Linf(\GG/\HH)\subset \Linf(\GG/\HH')$. On the other hand using Theorem \ref{quot} we get the converse containment: $\Linf(\GG/\HH')\subset \Linf(\GG/\HH )$. Thus $\Linf(\GG/\HH')= \Linf(\GG/\HH )$ and $\HH = \HH'$.
\end{proof}

\section{Group-like projections}

In this section we give a classification of those central projections $P\in Z(\Linf(\GG))$ that correspond to open quantum subgroups.
More precisely we prove $P\in Z(\Linf(\GG))$ is the support of an open quantum subgroup if and only if it is a so-called {\it group-like projection}.
This generalizes the similar result of Landstad and Van Daele  \cite{LandstadVD} for compact open subgroups of classical locally compact groups; note that group-like projections play also an essential role in \cite{FranzSkalski}, where they are used to study idempotent states on finite quantum groups.

\begin{definition}
Let $\GG$ be a locally compact quantum group.
A central projection $P\in Z(\Linf(\GG))$ is said to be a group-like projection in $\GG$ if
\begin{equation}\label{1}
\Delta^{\GG}(P)(\I\otimes P) = P\otimes P .
\end{equation}
\end{definition}

Group-like projections automatically satisfy certain additional properties.

\begin{lemma}\label{thmonP}
Let $\GG$ be a locally compact quantum group and let $P$ be a group-like projection in $\GG$.
Then
\begin{enumerate}
\item
$P\in\M(\C_0(\GG))$;
\item
$\tau^{\GG}_t(P) = P$;
\item
$R^{\GG}(P) = P$;
\item
$\Delta^{\GG}(P)(P\otimes \I) = P\otimes P$.
\end{enumerate}
\end{lemma}
\begin{proof}
Since
\[
P\otimes P = \Delta^\GG(P)(\I\otimes P) \in\M(\C_0(\GG)\otimes \mathcal{K}(\Ltwo(\GG)),
\]
we conclude $P\in\M(\C_0(\GG))$.

Since $P\in Z(\Linf(\GG))$, it follows $\sigma^\GG_t(P) = P$ for all $t\in \mathbb R$, and
\[\begin{split}
\tau^\GG(P)\otimes P &=\tau^\GG_t(P)\otimes\sigma^\GG_t(P)=(\tau^\GG_t\otimes\sigma^\GG_t)(\Delta^\GG(P)(\I\otimes P))\\&=\Delta^\GG(\sigma^\GG_t(P))(\I\otimes \sigma^\GG_t(P)))
=\Delta^{\GG}(P)(\I\otimes P) =P\otimes P,
\end{split}\]
which implies assertion (2).

To prove (3), first note that part (2) yields $R^\GG(P) = S^\GG(P)$.
Now, let $a, b \in \mathcal{N}_{\varphi^\GG}$. Applying the strong invariance of $\varphi^\GG$ we get
\[\begin{split}
(\id\otimes & \varphi^\GG)((\I\otimes a^*)\Delta^\GG(b)(P\otimes P)) = (\id\otimes\varphi^\GG)((\I\otimes Pa^*)\Delta^\GG(Pb))\\
&=S^\GG((\id\otimes\varphi^\GG)(\Delta^\GG( Pa^*)(\I\otimes Pb)))
=S^\GG((\id\otimes\varphi^\GG)(\Delta^\GG( a^*)\Delta^\GG(P)(\I\otimes Pb)))
\\
&=S^\GG((\id\otimes\varphi^\GG)(\Delta^\GG( a^*)(P\otimes Pb)))
=
S^\GG((\id\otimes\varphi^\GG)(\Delta^\GG(a^*)(\I\otimes Pb))P) \\&=
S^\GG(P)S^\GG((\id\otimes\varphi^\GG)(\Delta^\GG( a^*)(\I\otimes Pb)))=
S^\GG(P)(\id\otimes\varphi^\GG)((\I\otimes a^*)\Delta^\GG(b)\Delta^\GG(P))\\&=
(\id\otimes\varphi^\GG)((\I\otimes a^*)\Delta^\GG(b)\Delta^\GG(P)(S^\GG(P)\otimes \I)) \\&=
(\id\otimes\varphi^\GG)((\I\otimes a^*)\Delta^\GG(b)\Delta^\GG(P)(R^\GG(P)\otimes \I)) .
\end{split}
\]
By a standard approximation argument, using the fact that $\varphi^\GG$ is semifinite we conclude
\begin{equation}\label{toproof}
P\otimes P = \Delta^\GG(P)(R^\GG(P)\otimes \I) .
\end{equation}
Then it follows
\[
P\otimes P = \Delta(P)(R^\GG(P)\otimes \I) = (P\otimes P) (R^\GG(P)\otimes P) = PR^\GG(P)\otimes P .
\]
Thus $P = P R^\GG(P)$, and therefore
\[
P \leq R^\GG(P) .
\]
Applying $R^\GG$ to both sides of the above inequality we also get $P\geq R^\GG(P)$, and (3) follows.

Recall $\chi$ denotes the flip map. Now part (3) implies
\[\begin{split}
P\otimes P &= \chi\circ(R\otimes R) (P\otimes P)
=
\chi\circ(R\otimes R) (\Delta^{\GG}(P)(\I\otimes P))
\\&=
(R(P)\otimes \I) {\Delta^{\GG}}(R(P)) =
 (P\otimes \I){\Delta^{\GG}}(P),
\end{split}\]
which gives (4) since $P$ is central.
\end{proof}

We say that two closed quantum subgroups, $\HH_1$, $\HH_2$ of a locally compact quantum group $\GG$ (identified via morphisms $\pi_1 \in \Mor(\C_0^u(\GG), \C_0^u(\HH_1))$ and $\pi_2 \in \Mor(\C_0^u(\GG), \C_0^u(\HH_2))$ respectively) are \emph{isomorphic as quantum subgroups of $\GG$} if there exists an isomorphism $\rho:\C_0^u(\HH_1) \to \C_0^u(\HH_2)$, intertwining the respective coproducts, such that $\pi_2 = \rho \circ \pi_1$.

The following is the main result of this section.
\begin{theorem}\label{1-1}
Let $\GG$ be a locally compact quantum group.
There is a 1 to 1 correspondence between (isomorphism classes) of open quantum subgroups of $\GG$ and
 group-like projections in $\GG$.
\end{theorem}
\begin{proof}
Suppose $\HH$ is an open quantum subgroup of $\GG$. Then by Proposition \ref{delta(P)} the projection $\I_\HH$ is a group-like projection.

Conversely, suppose $P\in\Linf(\GG)$ is a group-like projection in $\GG$.
Let $\sN = P\Linf(\GG)$, then
\[
\Delta^{\sN}(x) := \Delta^\GG(x)(P\otimes P)
\]
defines a comultiplication on $\sN$.
We further define n.s.f. weights $\varphi^{\sN}= \varphi^\GG|_{\sN}$ and $\psi^{\sN} = \psi^\GG|_{\sN}$ on $\sN$.
Then direct computation as in the proof of Theorem \ref{lemh} shows $\varphi^{\sN}$ and $\psi^{\sN}$ are left and right invariant for $(\sN,\Delta^{\sN})$, respectively.
In particular there exists a locally compact quantum group $\HH$ such that $\sN = \Linf(\HH)$. Then the map $\pi:\Linf(\GG)\rightarrow \Linf(\HH)$,
$\pi(x) = xP$, identifies $\HH$ with an open subgroup of $\GG$. The facts that isomorphic open subgroups yield identical group-like projections and different group-like projections cannot yield isomorphic subgroups are easy to check.
\end{proof}

The next result connects further the considerations of this paper to these of \cite{LandstadVD} and \cite{LandstadVD2}.

\begin{proposition}
Let $\GG$ be a locally compact quantum group and $\HH$ an open quantum subgroup of $\GG$. 
Then $\HH$ is compact if and only if $\I_\HH\in\C_0(\GG)$.
\end{proposition}
\begin{proof}
Let $\pi: \Linf(\GG) \to \Linf(\HH)$ be the morphism identifying $\HH$ as an open subgroup of $\GG$.
By Corollary \ref{cor:surjecC0} we have $\C_0(\HH) = \pi(\C_0(\GG)) = \I_\HH \C_0(\GG)$, which allows to identify $\C_0(\HH)$ as a \cst-subalgebra of $\C_0(\GG)$.
Therefore, if $\HH$ is compact, then $\I_\HH \in \C_0(\HH) \subset \C_0(\GG)$.

Conversely, assume $\I_\HH \in \C_0(\GG)$. 
Then by Corollary \ref{cor:surjecC0}, $\I_{\HH} = \pi(\I_\HH) \in \C_0(\HH)$, hence $\HH$ is compact.
\end{proof}

Finally we record one more fact, obvious in the classical context.

\begin{proposition}
Let $\GG$ be a locally compact quantum group and $\HH$ a discrete   quantum group. If $\HH$ is an open subgroup of $\GG$ then $\GG$ is a discrete quantum group.
\end{proposition}
\begin{proof}
Using Remark \ref{redinj} we may view $\C_0(\hat\HH) = \C(\hat\HH)$ as a nondegenerate subalgebra of $\C_0(\hat\GG)$. Since $\C(\hat\HH)$ is unital, $\hat\GG$ is compact and hence $\GG$ is discrete.
\end{proof}

\section{Representation theoretic characterization of open quantum subroups} \label{Sec:BKLS}
The goal of this section is to prove a quantum version of the classical result of Bekka, Kaniuth, Lau and Schlichting \cite{BKLS}
which gives a characterization of open subgroups as those closed subgroups whose full group $\cst$-algebras embed injectively in the full group $\cst$-algebra of the ambient group.
This in particular links the representation theory of a locally compact quantum group with its open quantum subgroups and leads a way to a simpler picture of the induction theory (see \cite{rief1} and \cite{KKSinduced}).

We first record a disintegration result of the Haar weight through an open quantum subgroup (see \cite[Proposition 2.4]{Fima} for an analogous result in the discrete setting).
Let $\HH$ be an open quantum subgroup of $\GG$.
Recall from \eqref{HactG} and Corollary \ref{standardhomog} that the canonical action $\alpha$ of $\HH$ on $\GG$ is defined by the formula
\[
\alpha(x) = \Delta^\GG(x)(\I\otimes \I_\HH) \hspace{0.5cm}, x\in \Linf(\GG).
\]

\begin{proposition}\label{quo-weight}
Let $\GG$ be a locally compact quantum group, and let $\HH$ be an open quantum subgroup of $\GG$.
Then there exists a unique n.s.f. weight $\theta^{\GG/\HH}$ on $\Linf(\GG/\HH)$ such that
\[
(\theta^{\GG/\HH}\otimes\varphi^{\HH})\circ\alpha = \varphi^\GG .
\]
\end{proposition}

\begin{proof}
Applying Theorem \ref{lemh} we get
\[
\alpha(\delta^\GG)  = \Delta^\GG(\delta^\GG)(\I\otimes \I_\HH) = \delta^\GG\otimes \delta^\GG \I_\HH = \delta^\GG\otimes\delta^{\HH} .
\]
Hence the existence of $\theta^{\GG/\HH}$ follows from \cite[Proposition 8.7]{KustermansInduced}.
\end{proof}

In next theorem, we prove the forward implication of the Bekka-Kaniuth-Lau-Schlichting's characterization of open subgroups in the quantum setting,
namely, we prove open quantum subgroups have the above mentioned embedding property of full \cst-algebra.

\begin{theorem}\label{main}
Let $\GG$ be a coamenable locally compact quantum group, and let $\HH$ be an open quantum subgroup of $\GG$.
Then the dual morphism $\hat\pi\in\Mor(\C_0^u(\hat\HH), \C_0^u(\hat\GG))$ is injective and $\hat\pi(\C_0^u(\hat\HH))\subset \C_0^u(\hat\GG)$.
\end{theorem}
\begin{proof}
The inclusion $\hat\pi(\C_0^u(\hat\HH))\subset \C_0^u(\hat\GG)$ was proved in Lemma \ref{emded}. We thus need to prove the injectivity of $\hat{\pi}$. To this end we will use the notion of \emph{positive definite functions} on $\GG$, studied in \cite{DawsSalmi}. We first build a big enough family of `test-functionals', allowing us to verify positive-definiteness.

Recall that given $f\in\Lone_{\#}(\GG)$ we define $f^\#\in\Lone_{\#}(\GG)$ by
\[f^\#(x) = \overline{f(S^\GG (x)^*)} ,\;\; x \in \Dom (S^{\GG}).\]
By \cite[Proposition 4.4.4]{VaesPhD}, $\mathcal{N}_\varphi \cap \mathcal{N}_\psi$ is a core for both $\Lambda_\varphi$ and $\Lambda_\psi$.
For $x \in \mathcal{N}_\varphi \cap \mathcal{N}_\psi$ and ${m, n, k } \in \mathbb N$ we define
\[
x_{m, n, k} \,=\, \frac{mnk}{\pi^{\frac32}}\,\iiint\,e^{-m^2t_1^2 - n^2t_2^2 - k^2t_3^2}\,\sigma^\varphi_{t_1}(\sigma^\psi_{t_2}(\tau_{t_3}(y)))\,dt_1 \,dt_2 \,dt_3 .
\]
Then $x_{m, n, k} \in \mathcal{N}_\varphi \cap \mathcal{N}_\psi$, and
since the automorphism groups $\sigma^\varphi$, $\sigma^\psi$ and $\tau$ commute pairwise,
$x_{m, n, k}$ is analytic with respect to $\sigma^\varphi$, $\sigma^\psi$ and $\tau$.
Moreover, the set $\mathcal{F} := \{x_{m, n, k} : x \in \mathcal{N}_\varphi \cap \mathcal{N}_\psi, {m, n, k } \in \mathbb N\}$ is dense in $C_0(\GG)$.

Let $b, c \in \mathcal{F}$. Then by \cite[Lemma 3]{DawsSalmi} $b\sigma_{-i}(c)\varphi = b \varphi c \in \Lone_{\#}(\GG)$,
and
$$(b\sigma_{-i}(c)\varphi)^\# = (\tau_{\frac{i}{2}}(b) \varphi \tau_{-\frac{i}{2}}(c))\circ R = S(c) \psi S^{-1}(b) .$$
Since $\tau_z(x)$ is analytic with respect to $\sigma^\psi$ for all $z\in\mathbb C$ and $x\in\mathcal{F}$,
we conclude that
$$(b\sigma_{-i}(c)\varphi)^\# = S(c) \sigma_{-i}(S^{-1}(b)) \psi .$$
Hence, the set $\mathcal C := \{\omega \in \Lone_{\#}(\GG) : \omega = a\varphi \text{ for some } a\in C_0(\GG), \text{ and } \omega^\#= d\psi \text{ for some } d\in C_0(\GG)\}$
is dense in $\Lone_{\#}(\GG)$ with respect to the norm of $\Lone(\GG)$.
Since $\mathcal C$ is invariant under $\tau$, it follows by a similar argument to that of \cite[Lemma 3]{DawsSalmi} that $\mathcal C$
is dense in $\Lone_{\#}(\GG)$ with respect to its natural norm (i.e.\ the norm given by the maximum of $\Lone(\GG)$ norms of $f$ and $f^{\#}$).

Let then $f \in \mathcal C$, $f = a\varphi^\GG$ for some $a\in \C_0 (\GG)$. We have then for all $x \in \Dom(S)$, recalling that $\lambda>0$ denotes the scaling constant of $\GG$,
\[
\begin{split}
f^{\#}(x) &= \overline{\varphi^\GG( S^\GG (x)^*a)}  =\varphi^\GG(a^* S^\GG (x))=\varphi^\GG\circ S^\GG ( (S^\GG)^{-1}(a^* S^\GG (x))) \\&= \varphi^\GG\circ S^\GG  (x (S^\GG)^{-1}(a^*)) =  (\lambda^{\frac{i}{2}} (S^\GG)^{-1}(a^*)\psi^\GG)(x).
\end{split}
\]

By density and boundedness of the functionals involved the above formula holds for all $x \in \C_0(\GG)$. Now let $\mathcal{V}\in\M(\C_0(\HH)\otimes\cK(\sH))$ be a unitary representation of $\HH$ on a Hilbert space $\sH$, and let
$x = (\id\otimes\omega_{\zeta})(\mathcal{V}) \in \C_0(\HH) \subseteq \C_0(\GG)$
for a unit vector $\zeta\in H$. Then, for $f$ as above

\[
 \begin{split}
f^{\#}*f(x)
&= \lambda^{\frac{i}{2}}(\psi^\GG\otimes\varphi^\GG)(\Delta^\GG(x)( (S^\GG)^{-1}(a^*)\otimes a))\\
&=\lambda^{\frac{i}{2}} \psi^\GG((\id\otimes\varphi^\GG)(\Delta^\GG(x)(\I\otimes a)) (S^\GG)^{-1}(a^*))\\
& = \lambda^{\frac{i}{2}} \psi^\GG( (S^\GG)^{-1}( S^\GG ((\id\otimes\varphi^\GG)(\Delta^\GG(x)(\I\otimes a)))) (S^\GG)^{-1}(a^*)) \hspace{0.5cm} \text{ \Big(strong left inv. of $\varphi^\GG$\Big)}\\
& =  \lambda^{\frac{i}{2}} \psi^\GG( (S^\GG)^{-1}((\id\otimes\varphi^\GG)((\I\otimes x)\Delta^\GG(a))) (S^\GG)^{-1}(a^*))\text{ \Big(scaling constant relation  \Big)} \\ \hspace{0.5cm}
& = \varphi^\GG(a^*(\id\otimes\varphi^\GG)((\I\otimes x)\Delta^\GG(a)))\\
& = (\varphi^\GG\otimes\varphi^\GG)((a^*\otimes x)\Delta^\GG(a))\\
& = (\varphi^\GG\otimes\varphi^\HH)((a^*\otimes x)\alpha (a))   \hspace{2.8cm} \Big(x\in\C_0(\HH)\Rightarrow(\I\otimes x)\Delta^\GG(a) = (\I\otimes x)\alpha(a)\Big) \\
& = (\varphi^\GG\otimes\varphi^\HH\otimes\omega_{\zeta})((a^*\otimes \I)\mathcal{V}_{23}\alpha (a)_{12})\\
& = (\theta^{\GG/\HH}\otimes\varphi^\HH\otimes\varphi^\HH\otimes\omega_{\zeta})(\alpha(a^*)_{12}\mathcal{V}_{34}((\alpha\otimes\id)(\alpha(a)))_{123})  \hspace{2cm} \text{ \Big(by Proposition \ref{quo-weight}\Big)}\\
& = (\theta^{\GG/\HH}\otimes\varphi^\HH\otimes\varphi^\HH\otimes\omega_{\zeta})(\alpha(a^*)_{12}\mathcal{V}_{34}((\id\otimes\Delta^\HH)(\alpha(a)))_{123})\\
& = (\theta^{\GG/\HH}\otimes\varphi^\HH\otimes\varphi^\HH\otimes\omega_{\zeta})(\alpha(a^*)_{12}\mathcal{V}_{24}^*(\id\otimes\Delta^\HH\otimes\id)(\mathcal{V}_{23}\alpha(a)_{12}))
\text{ \Big(as  } \mathcal{V} \text{ is a representation \Big)}\\
& = (\theta^{\GG/\HH}\otimes\omega_{\zeta})((\id\otimes\varphi^\HH\otimes\id)(\mathcal{V}_{23}\alpha(a)_{12})^*(\id\otimes\varphi^\HH\otimes\id)(\mathcal{V}_{23}\alpha(a)_{12}))\geq 0 .
 \end{split}
\]

By density of $\mathcal{C}$ in $\Lone_{\#}(\GG)$ we can deduce now that a positive definite element $x\in\C_0(\HH)$ is positive definite when viewed as an element of $\C_0(\GG)$.  Let us consider  $\hat\pi \in\Mor( \C_0^u(\hat\HH),\C_0^u(\hat\GG))$.
Since $\GG$ is coamenable \cite[Theorem 15]{DawsSalmi} implies the adjoint map $\hat\pi^*: \C_0^u(\hat\GG))^*\rightarrow\C_0^u(\hat\HH))^*$ is surjective,
hence $\hat\pi$ is injective.
\end{proof}

We now prepare to prove the converse of the above theorem.
First recall from \cite[Definition 1.3, Theorem 1.6]{KV} that
\[
\mathcal{G} = \{\alpha\omega : \alpha\in(0,1), \,\omega \in \C_0(\GG)^*,\, 0\leq\omega(x)\leq \varphi^\GG(x) \textrm{ for all } x\in \C_0(\GG)^+\}
\]
is a directed set, and
\[
\varphi^\GG(x) = \lim_{\omega\in\mathcal{G}}\omega(x)
\]
for all $x\in\C_0(\GG)^+$.
We also need to use the following notion of \emph{properness} for quantum group actions which was studied independently by Kustermans \cite{KusterWeights} (for quantum groups) and Rieffel \cite{rief} (for  classical groups).

\begin{definition}\label{properaction}
Let $\alpha\in\Mor(\sA,\sA\otimes\C_0(\HH))$ be an action of a locally compact quantum group
$\HH$ on a $\C^*$-algebra $\sA$; recall   that it satisfies the Podle\'s condition
\[
\alpha(\sA)(\I\otimes\C_0(\HH)) = \sA\otimes\C_0(\HH) .
\]
We say
\begin{itemize}
\item[(i)]
$a\in \sA^+$ is $\alpha$-integrable if the net
\[
\big((\id\otimes\omega)\alpha(a)\big)_{\omega\in\mathcal{G}}
\]
converges strictly to an element $b\in\M(\sA)$;
\item[(ii)]
the action $\alpha$ is proper if the set
\[
\mathcal{P}_\alpha = \{a\in \sA^+: a \textrm{ is } \alpha-\textrm{integrable}\}
\]
is dense in $A^+$.
\end{itemize}
\end{definition}
It follows from the results of \cite[Section 2]{rief} that $\mathcal{P}_\alpha$ is a hereditary cone.

Classically an action of a closed subgroup on the ambient locally compact group is  always proper. Below we note this is also true for open quantum subgroups.

\begin{lemma}\label{propal}
Let $\GG$ be a locally compact quantum group and $\HH\leq\GG$ an open quantum subgroup.
Then the action
of $\HH$ on $\C_0(\GG)$ is proper.
\end{lemma}
\begin{proof}
By \cite[Result 2.4 and Lemma 1.21]{KV} the comultiplication $\Delta^\GG$ is a proper action of $\GG$ on $\C_0(\GG)$.
Under the identification $\C_0(\HH) = \I_\HH\C_0(\GG)$, $\alpha(x) = \Delta^\GG(x)(\I\otimes \I_\HH)$ for all $x \in \C_0(\GG)$ (essentially by Corollary \ref{standardhomog}), so properness of $\alpha$ follows from \cite[Result 1.19, part (3)]{KV}.
\end{proof}

We need another lemma characterising openness of a quantum subgroup via a corresponding group-like projection in a slightly different manner. It should be compared to Theorem \ref{minimal}.

\begin{lemma}\label{lift}
Let $\HH$ be a locally compact quantum group identified with a closed (in the sense of Woronowicz) subgroup of $\GG$ via (the reduced morphism) $\pi_r\in\Mor(\C_0(\GG),\C_0(\HH))$. Then $\HH$ is open in the sense of Definition \ref{defm} if and only if there exists  a projection $P\in\M(\C_0(\GG))$ such that for all $x \in \C_0(\GG)$
\[
x\in\ker \pi_r\iff xP = 0 .
\]
\end{lemma}

\begin{proof}
It was shown in Lemma \ref{thmonP} that if $\HH \leq \GG$ is an open quantum subgroup, then $\I_\HH$
satisfies the required properties.

For the converse, first note that since $(\I-P) P = 0$, by the assumption we have $\pi_r(\I-P) = 0$, i.e. $\pi_r(P) =\I$.
This gives
\[
\pi_r((\I-P)x) = \pi_r(x) - \pi_r(P)\pi_r(x) = 0
\]
for all $x\in\C_0(\GG)$, and therefore by the assumption
$(\I-P)xP = 0$. Equivalently, $xP =PxP$, and we get
\[
Px = (x^*P)^*= (Px^*P)^* = PxP = xP
\]
for all $x\in\C_0(\GG)$. Hence $P$ is a central element.

Thus, since $\pi(\C_0(\GG)) = \C_0(\HH)$ (cf. \cite[Theorem 3.6]{DKSS}) we may identify $\C_0(\HH)$ with $\sA = P\C_0(\GG)$.
Under this identification $\pi_r(x) = Px$ for all $x\in\C_0(\GG)$, and
$\psi^\HH = \psi^\GG|_{\sA}$.

Passing to the GNS representations of $\psi^\HH$ and $\psi^\GG|_{\sA}$, we then identify $\Linf(\HH)$ with $P\Linf(\GG)$,
and obtain the normal *-homomorphism
\[
\pi:\Linf(\GG)\rightarrow\Linf(\HH)
\]
satisfying conditions of Definition \ref{defm}.
\end{proof}

\begin{example}
Using Lemma \ref{lift} we can employ   Rieffel deformation of locally compact   groups to get examples of open quantum subgroups. To be more precise, let $\Gamma$, $H$ and $G$ be locally compact groups such that
\begin{itemize}
\item $H$ is an open subgroup of $G$;
\item $\Gamma$ is a closed subgroup of $H$;
\item $\Gamma$ is   abelian.
\end{itemize}
Let $\Psi$ be a $2$-cocycle on $\hat\Gamma$.
Then as noted in \cite{pkaspRieff}, Rieffel deformation $\HH^\Psi$ of $H$ may be viewed as a closed quantum subgroup of $\GG^\Psi$ in the sens of Vaes. In particular we have a morphism $\pi\in\Mor(\C_0(\GG^\Psi),\C_0(\HH^\Psi))$ which identifies $\HH^\Psi$ with a closed quantum  subgroup of  $\GG^\Psi$ in the sense of Woronowicz. Let us note that  $\I_H\in\M(\C_0(\GG))$ is an invariant element under the left and the right shifts by $\Gamma\subset H$. Using the description of $\GG^\Psi$ in terms of   crossed products given  in   \cite{RDVCP},  we can view  $\I_H$ as an element of $\M(\C_0(\GG^\Psi))$ and prove that $x\in \ker\pi$ iff $\I_H x = 0$. Thus the Lemma \ref{lift} implies that  $\HH^\Psi$ is open quantum subgroup of $\GG^\Psi$ and  $  \I_H = \I_{\HH^\Psi}$.

To be more specific, let us consider  $G$  the Lorentz group $O(1,3)$ and $H$ the  proper Lorentz group $SO(1,3)^+$. Then $H$ is an open subgroup of $G$. In order to describe $\Gamma$ let us consider  $\pi:SL(2,\mathbb{C})\to SO(1,3)^+$ the standard  two-fold covering. The description of Rieffel deformation of $SL(2,\mathbb{C})$ based on    $\Gamma_0 \subset  SL(2,\mathbb{C})$
\[\Gamma_0 = \left\{\begin{bmatrix}1&z\\0&1 \end{bmatrix}:z\in\mathbb{C}\right\}
\] was given in \cite{QHLG}. Defining $\Gamma = \pi(\Gamma_0)$ we get $\Gamma\subset SO(1,3)^+\subset O(1,3)$ and   Rieffel deformation applied to this case yields an example of open quantum subgroup   $\HH^\Psi \subset \GG^\Psi$, where $\HH^\Psi, \GG^\Psi$ are Rieffel deformations of $H$ and $G$ respectively  and $\Psi$ is the 2-cocyle used in \cite{QHLG}.
\end{example}

We need one more lemma before stating our next theorem and completing our characterization of open quantum subgroups as described at the beginning of this section.
It is likely well known, but we provide the proof for the convenience of the reader.

\begin{lemma} \label{compacts}
Let $\sA$ be a \cst-algebra which is non-degenerately represented on a Hilbert space $\sH$ and assume that $\sA$ is isomorphic to a \cst-subalgebra of the algebra $\cK(\sK)$ for some Hilbert space $\sK$. Then the von Neumann algebra $\sA''\subset B(\sH)$ and the multiplier algebra $\M(\sA)$, also viewed as a subalgebra of $B(\sH)$, are equal.
\end{lemma}

\begin{proof}
As $\sA$ is isomorphic to a $\cst$-subalgebra of compact operators, it is of the form $\oplus_{i \in \Jnd} \cK (\sK_i)$ for some family of Hilbert spaces $(\sK_i)_{i \in \Jnd}$ (\cite[Section 1.4]{Arvesonbook}). Thus we can view the embedding $\sA \subset B(\sH)$ as a nondegenerate representation $\pi$ of $\oplus_{i \in \Jnd} \cK (\sK_i)$, which is up to a unitary equivalence (which does not affect the statement we are proving) given by a direct sum of amplified identity representations,
$\pi=\oplus_{i \in \Jnd} \id \ot I_{\sL_i}$ (again see \cite[Section 1.4]{Arvesonbook}), where $\sL_i$ are auxiliary Hilbert spaces (in particular, ignoring the unitary equivalence, we have $\sH=\oplus_{i \in \Jnd} \sK_i \ot \sL_i$). Now as both the operations of passing to the concretely represented multiplier algebra and to the von Neumann algebraic closure exchange direct sums into direct products, it suffices to observe that for each $i \in \Jnd$ if we consider $\sA_i:= \cK(\sK_i) \ot I_{\sL_i} \subset B(\sK_i \ot \sL_i)$ we obviously have $\M(\sA_i) = B(\sK_i) \ot I_{\sL_i} = \sA_i ''$.
\end{proof}

\begin{theorem}\label{qb}
Let $\GG$ be a regular locally compact quantum group and $\HH$  a closed quantum subgroup of $\GG$, identified via the morphism
\[
\hat\pi:\Linf(\hat\HH)\rightarrow\Linf(\hat\GG),
\]
which satisfies the property $\hat\pi(\C_0(\hat\HH))\subseteq\C_0(\hat\GG)$.
Further assume that
 $\pi\in\M(\C_0^u(\GG),\C^u_0(\HH))$ admits a reduced version $\pi_r\in\M(\C_0(\GG),\C_0(\HH))$.
  Then $\HH$ is an open quantum subgroup of $\GG$.
\end{theorem}
\begin{proof}
In the course of the proof we shall only use the reduced version $\pi_r$ of $\pi$, so for simplicity we shall denote it by $\pi\in\Mor(\C_0(\GG),\C_0(\HH))$.

Let $\C_0(\GG/\HH)$ be the quantum homogeneous space in the sense of \cite[Theorem 6.1]{Vaes-induction} (see Definition \ref{defhomspace}) (remember we assume $\GG$ is regular).
By \cite[Corollary 6.4]{Vaes-induction}, $\C_0(\GG/\HH)$ is Morita equivalent to $\C_0(\GG)\rtimes \HH$.
Then, as $\hat{\pi}$ allows us to identify $\C_0(\hHH)$ with a subalgebra of $\C_0(\hGG)$, using regularity of $\GG$ gives
\begin{equation}\label{mor-eq}
\C_0(\GG)\rtimes\HH\subseteq\C_0(\GG)\rtimes\GG = \cK(\Ltwo(\GG))
\end{equation}
(the first inclusion formally speaking is an isomorphism onto a \cst-subalgebra).
In particular $\C_0(\GG/\HH)$  may be identified with a subalgebra of $\cK(\Ltwo(\GG))$.

Thus, in view of Lemma \ref{compacts} we have the identification $\M(\C_0 (\GG/\HH))=\Linf(\GG/\HH)$.
By \cite[Proposition 4.7]{KaspSol} the quantum homogeneous space $\C_0(\GG/\HH)$ is a nondegenerate subalgebra of $\M(\C_0(\GG))$ and
\[
\pi|_{\C_0(\GG/\HH)}\in\Mor(\C_0(\GG/\HH),\C_0(\HH)) .
\]
The identification $\M(\C_0 (\GG/\HH))=\Linf(\GG/\HH)$ yields a central projection $P\in\M(\C_0(\GG/\HH))$ (the von Neumann algebraic central support of the corresponding normal extension of $\pi$) such that for $x\in\C_0(\GG/\HH)$ we have
\begin{equation}\label{P<->pi}
\pi(x) = 0\iff xP = 0 .
\end{equation}
Let us note that $P$ is minimal in $\Linf(\GG/\HH)$. Indeed, for all $x\in\Linf(\GG/\HH)$ we have
\[(\id\otimes\pi)\Delta^\GG(x) = x\otimes\I.\]
Thus $\Delta^\GG(x)(\I\otimes P) = x\otimes P$ for all $x\in\Linf(\GG/\HH)$. In particular for all $\omega\in\B(\Ltwo(\GG))_*$ we have $(\omega\otimes\id)(\Delta^\GG(x))P = \omega(x)P$. Since \[\Linf(\GG/\HH) = \overline{\{(\omega\otimes\id)(\Delta^\GG(x)):\omega\in\B(\Ltwo(\GG))_*, x\in\Linf(\GG/\HH)\}}^{\textrm{weak}}\] we conclude that $P$ is minimal in $\Linf(\GG/\HH)$. Using Theorem \ref{minimal} we see that $\HH$ is open in $\GG$.

 \end{proof}

\begin{remark}\label{rem2}
Note that in the classical case the proof ends as soon as we get the identification
\[
\C_0(G/H) \cong\oplus\cK(H_i) .
\]
Thus we get a different proof of \cite[Theorem 5.4]{BKLS}. The regularity, which is the only assumption used up to that point, in the classical case holds automatically.
\end{remark}

\section{Quantum homogeneous space}
A simple characterization of open subgroups $H$ among closed subgroups of a locally compact group $G$ in terms of the homogeneous spaces is the following:
$H\leq G$ is open if and only if the homogeneous space $G/H$ is discrete. To introduce the quantum counterpart of this result, it is natural to propose the following definition.

\begin{definition}
Let $\GG$ be a regular locally compact quantum group with a closed subgroup $\HH$.
We say the homogeneous space $\GG/\HH$ is discrete if $\C_0(\GG/\HH) \cong \oplus_{i\in \Jnd} \cK(\Hil_i)$ for a family of Hilbert spaces $(\Hil_i)_{i \in \Jnd}$,
and we say $\GG/\HH$ is finite if $\C_0(\GG/\HH)$ is a finite-dimensional \cst-algebra.
\end{definition}

It was essentially shown in the course of the proof of Theorem \ref{qb} that discreteness of the quantum homogeneous space
(under the technical assumption of the properness of the action) yields openness of the corresponding closed quantum subgroup. We now record the converse of this fact.

\begin{proposition}\label{C*-quot-dir-sum}
Let $\GG$ be a regular locally compact quantum group, and let $\HH$ be an open quantum subgroup of $\GG$.
Then  the homogeneous space $\GG/\HH$ is discrete.
\end{proposition}

\begin{proof}

By Theorem \ref{Vaescl}, $\HH$ is closed in the sense of Vaes, and by Remark \ref{redinj}, $\C_0(\hat\HH)\subset \C_0(\hat\GG)$.
Hence \eqref{mor-eq} holds, which implies
$\C_0(\GG/\HH)$ is isomorphic to a subalgebra of $\cK(\Ltwo(\GG))$. This completes the proof.
\end{proof}

In the compact case the result naturally simplifies, taking the following form.

\begin{corollary}
Let $\GG$ be a compact quantum group, and let $\HH$ be a closed quantum subgroup of $\GG$. Then $\HH$ is open in $\GG$ if and only if $\GG/\HH$ is finite.
\end{corollary}
\begin{proof}
Recall that every compact quantum groups is regular.
Therefore, if $\GG/\HH$ is finite then as remarked above it follows from the proof of Theorem \ref{qb} that $\HH$ is open.

Conversely, if $\HH$ is open then Proposition \ref{C*-quot-dir-sum} implies $\GG/\HH$ is discrete, and since moreover
$\C_0(\GG/\HH)$ is unital, the assertion follows.
\end{proof}

The above fact and \cite[Theorem 6.17]{Pinzari} imply that  a connected component of the identity of a compact quantum group $\GG$ of Lie type with commutative and normal torsion representation category $\textup{Rep} \GG^t$ is open in $\GG$ (we refer to \cite{Pinzari} for the corresponding terminology). It also implies that for example Woronowicz's $SU_q(2)$ does not admit any open quantum subgroup for all $q\in[-1,0)\cup(0,1]$, as follows from the list of the quantum subgroups in \cite{Podlessubgroups}.

Recall from the introduction that  in contrast to the von Neumann algebra $\Linf(\GG/\HH)$, the construction of $\C_0(\GG/\HH)$
is highly non-trivial, and not explicit (does not identify $\C_0(\GG/\HH)$ as a concrete \cst-algebra).
In fact, even its existence in general, without the regularity assumption, remains an open problem.

Our next theorem shows that the \cst-algebra $\C_0(\GG/\HH)$ has in fact a very concrete and simple realization when $\HH$ is open.

\begin{theorem}\label{C*-quot}
 Let $\GG$ be a regular locally compact quantum group and let $\HH \leq \GG$ be an open quantum subgroup. Then
\[
\C_0(\GG/\HH) = {\{(\omega\otimes\id) \Delta^\GG(\I_\HH):\omega\in \Lone(\GG)\}}^{-\|\cdot\|}.
\]
\end{theorem}
\begin{proof}
By Proposition \ref{C*-quot-dir-sum} we have $\C_0(\GG/\HH)\cong\oplus\cK(H_i)$, and therefore using Lemma \ref{compacts} we get
\begin{equation}\label{idlm1}
\Linf(\GG/\HH) = \M(\C_0(\GG/\HH)) \cong \prod_i \B(H_i).
\end{equation}

Denote $\sB := {\{(\omega\otimes\id) \Delta^\GG(\I_\HH):\omega\in\Lone(\GG)\}}^{-\|\cdot\|}$.
It follows from \eqref{charcomm} that
\[
(R^\GG(x)\omega\otimes\id)\Delta^\GG(\I_\HH) = \big((\omega\otimes\id)\Delta^\GG(\I_\HH)\big) \,x \in \C_0(\GG/\HH)
\]
for all $\omega \in \Lone(\GG)$ and $x\in \C_0(\GG/\HH)$. Since $C_0(\GG/\HH)$ is non-degenerate it follows further that $\sB \subseteq \C_0(\GG/\HH)$.
Moreover $\sB$ is weak* dense in $\Linf(\GG/\HH)$ by Theorem \ref{quot}. Thus we conclude from \eqref{idlm1} that $\sB = \C_0(\GG/\HH)$.
\end{proof}

In view of the above characterization it is natural to define the quantum homogeneous space by
$\C_0(\GG/\HH) := {\{(\omega\otimes\id) \Delta^\GG(\I_\HH):\omega\in\Lone(\GG)\}}^{-\|\cdot\|}$ for a general (not necessarily regular)
locally compact quantum group $\GG$ and an open quantum subgroup $\HH$.
An immediate question then is whether in this case the set ${\{(\omega\otimes\id) \Delta^\GG(\I_\HH):\omega\in\Lone(\GG)\}}^{-\|\cdot\|}$
is isomorphic to a \cst-subalgebra of compact operators (or whether it satisfies the conditions in Theorem \ref{defhomspace}).

\section{Normal open quantum subgroups}

In this section we consider open quantum subgroups that are also normal (see Definition \ref{def:normal}).
We first give a characterization of normality of an open quantum subgroup $\HH\leq \GG$ in terms of the corresponding projection $\I_\HH$. Recall that if $x \in \Linf(\GG)$ and $\nu, \mu \in \Lone(\GG)$ then we write $\mu * x := (\id \ot \mu) \Delta(x) \in \Linf(\GG)$, $x * \mu = (\mu \ot \id) (\Delta(x)) \in \Linf(\GG)$,
$\mu * \nu = (\mu \ot \nu) \circ \Delta \in \Lone(\GG)$. 

\begin{theorem}
Let $\GG$ be a locally compact quantum group of Kac type, and let $\HH$ be
an open quantum subgroup of $\GG$. Then the following are equivalent:
\begin{enumerate}
\item
$\HH$ is normal in $\GG$;
\item
$\omega*\I_\HH = \I_\HH*\omega$ for all $\omega\in \Lone(\GG)$.
\end{enumerate}
\end{theorem}

\begin{proof}
Suppose (2) holds, then by Theorem \ref{quot}, and its natural `right' version,
\[\begin{split}
\Linf(\GG/\HH)
&=
\{(\omega\otimes\id)\Delta^\GG(\I_\HH):\omega\in\Lone(\GG)\}''
\\&=
\{(\id\otimes\omega)\Delta^\GG(\I_\HH):\omega\in\Lone(\GG)\}''
=
\Linf(\HH \backslash \GG)
\end{split}\]
which implies $\HH$ is normal in $\GG$.

Now suppose $\HH$ is normal. Then since every locally compact quantum group of Kac type is regular, Proposition \ref{C*-quot-dir-sum} and remarks after Definition \ref{def:normal}
imply that $\GG/\HH$ is a discrete quantum group of Kac type.
Moreover, by Proposition \ref{1_H:counit}, $\I_\HH$, as an element of $\linf(\GG/\HH)$ is the minimal central projection associated to the counit (the last statement follows from an elementary computation).

Hence, the implication $(1)\Rightarrow (2)$ reduces to verification of (2) for the support projection of the counit:
specifically, we have to show that if $\KK$ is  a discrete quantum group of Kac type and $e \in \ell^\infty(\KK)$ is the support projection of the counit then
$\omega*e = e*\omega$ for all $\omega\in \ell^\infty(\KK)_*$.
But this follows from the facts that $e$ is the regular representation of the dual Haar state $\hat\varphi$, and the latter is a trace.
\end{proof}

Next, we prove there is a canonical 1-1 correspondence between normal open quantum subgroups of a locally compact quantum group
and normal compact quantum subgroups of its dual.

\begin{theorem}\label{11opcom}
Let $\GG$ be a locally compact quantum group.
There is a 1-1 correspondence between normal compact quantum subgroups $\KK\leq \GG$ and normal open quantum subgroups $\hHH\leq \hGG$ given by the short exact sequences \eqref{ex1} and \eqref{ex2}; in particular
\[
\hHH \approx \widehat{\GG/\KK}\, .
\]
\end{theorem}

\begin{proof}
Suppose $\hat\HH$ is an open quantum subgroup of $\hat\GG$.
By Proposition \ref{redinj} we get $\C_0(\HH)\subset \C_0(\GG)$, hence it follows from \cite[Theorem 5.4]{KaspSol} that $\KK$ is compact.

Conversely, suppose $\KK$ is a compact normal subgroup of $\GG$ and let $\HH = \GG/\KK$.
The formula
\[\mathbb{E} = (\id\otimes \psi^{\KK})\circ\alpha:\Linf(\GG)\rightarrow\Linf(\GG/\KK)\] defines a conditional expectation (see for example \cite{SS} for related considerations). Moreover, if $x\in\Linf(\GG)$ satisfies $\psi^\GG(x^*x)<\infty$ then
\begin{equation}\label{ce}\psi^\GG(\mathbb{E}(x^*x)) = \psi^\GG(x^*x)<\infty.\end{equation}
Indeed, noting that
\[
\begin{split}
(\I_{\Linf(\GG)}\psi^\GG\otimes\id)&\circ\alpha(x^*x) = (\psi^\GG\otimes\id\otimes\id)(\Delta^\GG\otimes\id)\circ\alpha(x^*x)\\
&=(\psi^\GG\otimes\id\otimes\id)(\id\otimes\alpha)\circ\Delta^\GG(x^*x)=\alpha(\psi^\GG(x^*x)\I_{\Linf(\GG)})\\&=
\psi^\GG(x^*x)(\I_{\Linf(\GG)}\otimes\I_{\Linf(\HH)})
\end{split}
\]
 we get
\[(\psi^\GG\otimes\id)(\alpha(x^*x)) = \psi^\GG(x^*x)\I_{\Linf(\HH)}\] and
 \eqref{ce} is proved.
Moreover \eqref{ce} shows that restricting the  Haar measure $\psi^\GG$ to $\Linf(\GG/\KK) $ we get a n.s.f. right invariant weight. In particular we may identify $\Ltwo(\GG/\KK)$ with a subspace of $\Ltwo(\GG)$.

In what follows we shall denote $\Ltwo(\HH) = \Ltwo(\GG/\KK)$.
Let $P:\Ltwo(\GG)\rightarrow \Ltwo(\GG)$ be the orthogonal projection onto $\Ltwo(\HH)$.
The normality of $\KK$ yields $\Delta^\GG(x)\in\Linf(\GG/\KK)\bar\otimes\Linf(\GG/\KK)$, for all $x\in \Linf(\GG/\KK)$. In particular for all $x,y\in\mathcal{N}_\psi$ we have
\[\Delta^\GG(x)(\I\otimes y)\in\mathcal{N}_{\psi\otimes\psi}\] Thus $(P\otimes P)\ww^\GG = \ww^\GG(P\otimes P)$.

The multiplicative unitary $\ww^\HH$ is a unitary operator acting on $\Ltwo(\HH) \otimes\Ltwo(\HH)$ such that
\[\ww^\HH = (P\otimes P)\ww^\GG|_{\Ltwo(\HH)\otimes\Ltwo(\HH)}\] Moreover the map
\begin{equation}\label{ident3}\pi:\Linf(\HH) \ni Px|_{\Ltwo(\HH)} \mapsto x\in\Linf(\GG)
\end{equation} yields the identification of $\hat\HH$ with a closed subgroup of $\hat\GG$. To be more precise,  an element $x\in \Linf(\HH)\subset\Linf(\GG)$ when acts on $\Ltwo(\HH)$ is given by $xP$, which explains \eqref{ident3}.
 In particular we have
\begin{equation}\label{piact}(\id\otimes\pi)(\ww^\HH) = (P\otimes \I) \ww^\GG|_{\Ltwo(\HH)\otimes\Ltwo(\GG)} \end{equation}

Let us define a completely positive normal map  $\rho:\Linf(\hat\GG)\rightarrow\B(\Ltwo(\HH))$
\[\rho(x) = Px|_{\Ltwo(\HH)}\] We have (see \eqref{piact})
\begin{equation}\label{ropi}(\rho\otimes\id)(\ww^\GG) = (\id\otimes\pi)(\ww^\HH)\end{equation}
In what follows we shall show that $\rho$ is a $*$-homomorphism. In order to do this let us fix $\omega,\mu\in\Lone(\GG)$ and $x,y\in\Linf(\hat\GG)$
\[\begin{split}
x&=(\id\otimes\omega)(\ww^\GG)\\
y&=(\id\otimes\mu)(\ww^\GG)
\end{split} \]
We compute
\[
\begin{split}
\rho(xy) &=\rho( (\id\otimes\omega)(\ww^\GG)(\id\otimes\mu)(\ww^\GG)) = \rho( (\id\otimes\omega*\mu)(\ww^\GG))\\
& = (\id\otimes\omega*\mu)((\id\otimes\pi)\ww^\HH)= (\id\otimes\omega\otimes\mu)((\id\otimes\Delta^\GG)(\id\otimes \pi)(\ww^\HH))
\\&=
(\id\otimes\omega\otimes\mu)((\id\otimes \pi\otimes\pi)(\id\otimes\Delta^\HH)(\ww^\HH))
= (\id\otimes\omega\circ\pi)(\ww^\HH)(\id\otimes\mu\circ\pi)(\ww^\HH)\\
& = \rho((\id\otimes\omega )(\ww^\GG))\rho((\id\otimes\mu )(\ww^\GG)) = \rho(x)\rho(y)
\end{split}
\] Since the set of slices of $\ww^\GG$ is dense in $\Linf(\hat\GG)$ we conclude that $\rho$ is a homomorphism. Clearly it also   star preserving.

The injectivity and normality of $\pi$ together with \eqref{ropi}  shows that $\rho(\Linf(\hat\GG)) = \Linf(\hat\HH)$.
Moreover,
\[
\begin{split}
(\Delta^{\hat\HH}\circ\rho\otimes\id)(\ww^\GG) &= (\Delta^{\hat\HH}\otimes\pi)(\ww^\HH) =
(\id\otimes\pi)(\ww^\HH)_{23}(\id\otimes\pi)(\ww^\HH)_{13}\\&=(\rho\otimes\rho\otimes\id)((\Delta^{\hat\GG}\otimes\id)(\ww^{\hat\GG}))
\end{split}\]
Thus we conclude that $\Delta^{\hat\HH}\circ\rho =(\rho\otimes\rho)\circ\Delta^{\hat\GG}$.
\end{proof}
In the next proposition we shall use the terminology and notation introduced in the proof of Theorem \ref{11opcom}.
\begin{proposition}
Let $\KK\subset \GG$ be a normal compact subgroup of $\GG$ and   $P:\Ltwo(\GG)\to\Ltwo(\GG)$   the projection onto $\Ltwo(\GG/\KK)$ and $\HH=\GG/\KK$ . Then $P = \I_{\hat\HH}$.
\end{proposition}
\begin{proof}
Recall from the proof of Theorem \ref{11opcom} that the surjection
$\rho:\Linf(\hat\GG)\to\Linf(\hat\HH)$ which identifies $\hat\HH$ as an open quantum subgroup of $\hat\GG$ is defined by
\[\rho(x) = Px|_{\Ltwo(\HH)} .\]
In particular, the map $x\mapsto PxP$ is a *-homomorphism on $\Linf(\hat\GG)$, and hence $P\in \Linf(\hat\GG)'$.
From $ J^{\hat\psi} \Linf(\GG/\KK) J^{\hat\psi} = R^\GG(\Linf(\GG/\KK)) = \Linf(\GG/\KK)$ we conclude that $ J^{\hat\psi} P  J^{\hat\psi} = P$,
which yields $P \in\Linf(\hat\GG)$ and we are done.
\end{proof}

As all closed quantum subgroups of duals of classical subgroups are normal (and arise as duals of quotients of the ambient group, see \cite[Theorem 5.1]{DKSS}), we have the following immediate corollary.

\begin{corollary}
Let $G$ be a locally compact  group. There is a 1-1 correspondence between open quantum subgroups of $\hh{G}$ and normal compact subgroups of $G$:
 a compact normal subgroup of $K\subset G $ yields an open subgroup $\hh{H}$ in $\hh{G}$ such that $H = G/K$.
\end{corollary}

We end the section with another characterization of open quantum subgroups which essentially follows from the proof of Theorem \ref{11opcom}.

\begin{theorem} \label{thm:weightsopen}
Let $\HH$ be a closed quantum subgroup of $\GG$ identified by an injective morphism  $\gamma:\Linf(\hat\HH)\to  \Linf(\hat\GG)$.
If $\psi^{\hat\GG}\circ\gamma$ defines  an n.s.f. weight on $\Linf(\hat\HH)$, then $\HH$ is an open quantum subgroup of $\GG$.
\end{theorem}
\begin{proof}
The assumption on $\psi^{\hat\GG}$ shows that $\psi^{\hat\HH} = \psi^{\hat\GG}\circ\gamma$ which
allows us to identify $\Ltwo(\hat\HH)$ with a subspace of $\Ltwo(\hat\GG)$. We denote by $P$ the projection onto $\Ltwo(\hat\HH)\subset \Ltwo(\hat\GG)$. Then we have
\[\ww^{\hat\HH} = (P\otimes P)\ww^{\hat\GG}|_{\Ltwo(\hat\HH)\otimes \Ltwo(\hat\HH)}\]
Using this description of $\Ltwo(\hat\HH)$ we conclude that $\gamma(Px|_{\Linf(\hat\HH)}) = x$ and
\[(\id\otimes\gamma)(\ww^{\hat\HH}) = (P\otimes \I)\ww^{\hat\GG}|_{\Ltwo(\hat\HH)\otimes \Ltwo(\hat\GG)}\]
Define the normal c.p. map
$\rho:\Linf(\GG)\to\B(\Ltwo(\hat\HH))$ by
\[\rho(x) = Px|_{\Ltwo(\hat\HH)}\]
Then
\[(\id\otimes\gamma)(\ww^{\hat\HH}) = (\rho\otimes\id)\ww^{\hat\GG}\] Following the techniques of the proof of Theorem \ref{11opcom} we may conclude that $\rho$ is a normal $*$-homomorphism such that $\rho(\Linf(\GG)) = \Linf(\HH)$ and
\[\Delta^\HH\circ\rho = (\rho\otimes\rho)\circ\Delta^\GG\]
Since $\rho$ is a $*$-homomorphism it follows $P\in \Linf(\GG)'$. Moreover, since
\[R^{\hat\GG}(\gamma(\Linf(\hat\HH))) = J^\psi\gamma(\Linf(\hat\HH)J^\psi= \gamma(\Linf(\hat\HH))\]
we get $J^\psi P J^\psi = P$. Since $P\in\Linf(\GG)'$ we conclude that $P\in\Linf(\GG)$. In particular $P = \I_\HH$.
\end{proof}

\bibliography{open1}{}
\bibliographystyle{plain}

\end{document}